\documentclass[11pt,a4paper]{article}

\usepackage[dvips]{graphicx}
\usepackage{epsfig}
\usepackage{amsfonts}
\usepackage{amssymb}
\usepackage{amsmath,amsthm}
\usepackage{geometry}
\usepackage{multicol}
\usepackage{algorithmic}
\usepackage{algorithm}
\usepackage{mdwlist}
\usepackage{subfigure}
\usepackage{psfrag}
\usepackage[usenames]{color}

\newtheorem{thm}{Theorem}[section]
\newtheorem{cor}[thm]{Corollary}
\newtheorem{lem}[thm]{Lemma}
\newtheorem{prop}[thm]{Proposition}
\newtheorem{ex}{Example}

\newtheorem{fact}[thm]{Fact}
\newtheorem{assu}[thm]{Assumption}
\theoremstyle{definition}
\newtheorem{defi}[thm]{Definition}

\newtheorem{rem}[thm]{Remark}

\let\oldmarginpar\marginpar
\renewcommand\marginpar[1]{\oldmarginpar[\raggedleft\footnotesize #1]%
{\raggedright\footnotesize #1}}

\renewcommand{\algorithmiccomment}[1]{\hfill //#1 } 
\newcommand{\inputbox}[1]{\iobox{Input:}{#1}}
\newcommand{\outputbox}[1]{\iobox{Output:}{#1}}
\newcommand{\iobox}[2]{\parbox[t]{0.7in}{\bf #1}\parbox[t]{\columnwidth-0.7in}{#2}}

\newcommand{\norm}[1]{\left\Vert#1\right\Vert}
\newcommand{\abs}[1]{\left\vert#1\right\vert}
\newcommand{\Real}{\mathbb R}
\newcommand{\real}{\mathbb R}

\newcommand{\epsLexVector}{\mbox{\boldmath$\epsilon$}}

\newcommand{\ERRE}{\mathcal{R}}
\newcommand{\GI}{\mathcal{G}}
\newcommand{\BI}{\mathcal{B}}
\newcommand{\KAPPA}{\mathit{K}}
\newcommand{\QU}{\mathbf{Q}}
\newcommand{\PI}{\mathbf P}
\newcommand{\CI}{\mathcal{C}}
\newcommand{\mbzero}{\mathbf 0}

\newcommand{\inv}{^{-1}}
\newcommand{\columnIndex}[1]{_{ \cdot #1}}
\newcommand{\rowIndex}[1]{_{  #1}}
\newcommand{\fullySemimo}{\ensuremath{E_0^f}}

\newcommand{\timeLP}[2]{\ensuremath{T_{\textit{LP}}(#1,#2)}}

\DeclareMathOperator{\cone}{cone}

\DeclareMathOperator{\redundancy}{redundancy}
\DeclareMathOperator{\adjacency}{adjacency}

\DeclareMathOperator{\INT}{int}
\DeclareMathOperator{\aff}{aff}
\DeclareMathOperator{\relINT}{rel\,int}
\DeclareMathOperator{\rank}{rank}

\begin{document}

\title{An output-sensitive algorithm for multi-parametric LCPs with sufficient matrices}%
\author{S. Columbano\\
  Institute for Operations Research\\
  ETH Zurich, Switzerland\\
  seba@ifor.math.ethz.ch\\
  \and
  K. Fukuda\\
  Institute for Operations Research and\\
  Institute of Theoretical Computer Science\\
  ETH Zurich, Switzerland\\
  fukuda@ifor.math.ethz.ch\\
  \and
  C.N. Jones\\
  Automatic Control Laboratory\\
  ETH Zurich, Switzerland\\
  cjones@ee.ethz.ch\\
} \date{July 15, 2008}
\maketitle

\begin{abstract}
This paper considers the multi-parametric linear complementarity problem
(pLCP) with sufficient matrices.  The main result is an algorithm to
find a polyhedral decomposition of the set of feasible parameters and
to construct a piecewise affine function that maps 
each feasible parameter to a solution of the associated LCP in
such a way that the function is affine over each cell 
of the decomposition.  The algorithm is output-sensive in the sense that
its time complexity is polynomial in the size of the input
and linear in the size of the output, when the problem is 
non-degenerate.  We give a lexicographic perturbation
technique to resolve degeneracy as well.  Unlike for the non-parametric
case, the resolution turns out to be nontrivial, and
in particular, it involves linear programming (LP) duality and
multi-objective LP.
\end{abstract}

\section{Introduction}
Given a real square matrix $M$ and a vector $q$, solving a linear complementarity problem (LCP) consists of
finding two nonnegative vectors $w$ and $z$ that satisfy the conditions
\begin{align}\label{eq:DefIntroLCP}
  w-Mz=q,\quad w\geq0,\quad z\geq0,\quad w^Tz=0\enspace.
\end{align}
This simply stated and well-studied problem has far-reaching applications that have been well-documented in
the literature. Rather than give a survey here, the interested reader is referred to the
books~\cite{murty_onlinebook,cottle}.

Several authors have studied the properties of various parametric versions of this problem
(e.g.~\cite{Jones_Morari,klatte,cottle,klatte_1985_pLCP_structure,danao_1997_On_pLCP,tammer_1996_pLCP,DeSDeM:93-69}),
but unless there are restrictions placed on the particular parametric LCP (pLCP) considered, it is in general
unrealistic to expect an efficient computational algorithm. We here study the class of pLCPs where the matrix
$M$ is sufficient\footnote{Sufficient matrices are defined in Section~\ref{section:sufficient_matrices}.} and
the right hand side (the vector $q$ in~\eqref{eq:DefIntroLCP}) is allowed to vary within a given affine
subspace $S$. The goal is then to compute functions $z(\cdot)$ and $w(\cdot)$ that map from the affine
subspace $S$ to a solution for pLCP~\eqref{eq:DefIntroLCP} whenever one exists.

This class of pLCP includes the important cases of linear and convex quadratic programs, where parameters
appear linearly in the cost and the right hand side of the constraints~\cite{murty_onlinebook}. In recent
years, there has been a great deal of interest in the control community in parametric programming due to the
fact that an important class of control algorithms for constrained linear systems, called model predictive
controllers (MPC), can be posed as parametric linear or quadratic programs. The offline solution of these
parametric problems results in an explicit representation of the optimal control action, which in some cases
allows the controller to be implemented on systems with sampling rates of milli- and micro-seconds instead of
the traditional seconds and minutes~\cite{seron:goodwin:dona:2000,johansen:peterson:slupphaug:2000,BMDP02a}. A
similar setup results when computing optimal policies in a dynamic programming framework for partially
observable Markov decision processes~\cite{cassandra_1998_pomdp}.

While parametric programming is widely used for sensitivity analysis, it is also applied in several other
applications. In~\cite{jones_2008_proj_mplp} it was shown that polyhedral projection can be reduced to
parametric linear programming in polynomial time and of course such projections have uses ranging from the
computation of invariant sets~\cite{blanchini99} and force closures~\cite{ponce95} to program
analysis~\cite{program_analysis_2007} and theorem proving~\cite{hooker1992}. Polyhedral vertex and facet
enumeration can also be posed as projection problems, and hence solved with the proposed pLCP
approach~\cite{handbook}, which as discussed below results in an output sensitive algorithm in the
non-degenerate case (although not the most efficient one for this purpose).

In~\cite{klatte} it was shown that if $M$ is a sufficient matrix and $S$ satisfies certain \emph{general
  position}\footnote{General position is defined in Section~\ref{sec:critical domains}.} assumptions, then
$z(\cdot)$ and $w(\cdot)$ are unique piecewise affine functions that are defined over a polyhedral partition
of a convex set. There is, however, no known efficient method of testing this general position assumption
\emph{a priori} and in fact, it is often not satisfied even in the simplest case when the pLCP models a
parametric linear program~\cite{Colin_pLP_Algorithm}.

This paper extends the result of~\cite{klatte} by removing the restrictive and untestable general position
assumption, allowing the algorithm to operate on any pLCP which is defined by a sufficient matrix and an
affine subspace.  This is achieved through a lexicographic perturbation technique, which has the effect of
symbolically shifting the affine subspace an infinitesimally small amount and into general position. We first
demonstrate that this perturbation always results in a problem that is in fact in general position and hence
has the favorable uniqueness and partitioning properties discussed above. The challenge then becomes one of
doing calculations in this perturbed space. The main optimization problem that arises as a result of the
perturbation is a linear program that is polynomially parameterized by a positive variable $\epsilon$. The
decision problem to be tackled is then the determination of the behavior of this parametric problem as the
parameter $\epsilon$ tends to zero. Section~\ref{section:algo_non_GP} discusses how this problem can be
converted into a multi-objective linear program, which can then be solved efficiently. The proposed technique
should be applicable to other algorithms that rely on lexicographic perturbation to handle degeneracy.

The resulting algorithm has the strong property that its complexity is polynomial in the size of the input
(the matrix $M$) and linear in the size of the output (the number of pieces in the piecewise-affine functions
$w$ and $z$). For this reason, we call the algorithm `output sensitive', although it should be noted that the
complexity of the functions $w$ and $z$ can be exponential in the worst case and that this complexity result
is for the lexicographically shifted affine subspace, which may be more complex than the unshifted case.

The reminder of the paper is organized as follows.  Section~\ref{section:pLCP} gives some basic notations and
a formal definition of the parametric LCP.  Section~\ref{section:sufficient_matrices} provides some useful
properties of pLCPs on sufficient matrices. Section~\ref{section:Algo_GP} then presents the proposed method
with a general position assumption, and then this is relaxed in Section~\ref{section:algo_non_GP} where the
lexicographic perturbation is introduced. Finally, Section~\ref{section:complexity} analyzes the complexity of
the algorithm.

\section{Parametric LCP, critical regions and their adjacency}\label{section:pLCP}

Let us first fix some useful notations for matrices.  For a matrix $A\in\real^{m\times n}$ and a column index
$j\in \{1,\dots,n\}$, $A\columnIndex{j}\in\real^m$ denotes the $j$-th column vector of $A$.  Similarly, for a
row index $i\in \{1,\dots,m\}$ $A\rowIndex{i}\in\real^{I\times n}$ denotes the $i$-th row vector of $A$.  For
a subset $J\subseteq\{1,\dots,n\}$, $A\columnIndex{J}\in\real^{m\times J}$ denotes the matrix formed by the
columns of $A$ indexed by $J$, and for a vector $v\in\real^n$, $v_J$ denotes the vector formed by the
components of $v$ indexed by $J$.  For $I\subseteq\{1,\dots,m\}$, we denote by $A\rowIndex{I}\in\real^{I\times
  n}$ the matrix formed by the rows of $A$ indexed by $I$.

Given a real square matrix $M$ and a vector $q$ of size $n$, the \emph{linear complementarity problem} (LCP)
is to find two nonnegative vectors $w$ and $z$ that satisfy
\begin{align}\label{eq:Def1LCP}
  w-Mz=q,\quad w\geq0,\quad z\geq0,\quad w^Tz=0\enspace.
\end{align}
In this paper, we consider the LCP~\eqref{eq:Def1LCP} where the right-hand side is allowed to vary within some
affine subspace. Specifically, the goal is to find two functions $w(\cdot)$ and $z(\cdot)$ that
solve~\eqref{eq:Def1LCP} over a given affine subspace $S$.
\begin{defi}\label{defn:pLCP}
  Let $Q \in \real^{n\times d}$ be a matrix of rank $d$, $q\in \real^n$ a vector and $M\in\real^{n\times n}$ a
  matrix of order $n$.  The functions $w(\cdot)$ and $z(\cdot)$ are a solution to the
  pLCP~\eqref{eq:pLCP_equations} if for every $\theta\in \Theta_f$, $w(\theta)$ and $z(\theta)$ satisfy the
  relations
  \begin{subequations}\label{eq:pLCP_equations}
    \begin{align}
      w(\theta)-Mz(\theta)&=q+Q\theta\enspace, \label{eq:pLCP_equatiom}\\
      w(\theta),z(\theta) &\geq 0\enspace, \label{eq:pLCP_nonNegativity} \\
      w(\theta)^Tz(\theta)&=0\enspace, \label{eq:pLCP_complementarity}
    \end{align}
  \end{subequations}
  where $\Theta_f\subseteq\real^d$ is the set of feasible parameters $\theta$, that is, those for which a
  solution to~\eqref{eq:pLCP_equations} exists.
\end{defi}

For the remainder of the paper we assume that the problem data $M$, $q$ and $Q$ are given and we define
$A\in\real^{n \times 2n}$ to be the matrix $\begin{bmatrix}I&-M\end{bmatrix}$. Consider the following system
of linear equality constraints in non-negative variables
\begin{equation} \label{eq:Basis} Ax=q \enspace, \ \ \ x\geq 0\enspace.
\end{equation}
A \emph{basis} is a set $B\subset \{1,2, \dots,2n\}$ such that $|B|=n$ and $rank(A\columnIndex{B})=n$;
$N:=\{1,\dots,2n\}\backslash B$ is its complement and we call $x_B$ and $x_N$ the \emph{basic} and
\emph{non-basic} variables respectively.  Every basis $B$ defines a \emph{basic solution} to the linear system
(\ref{eq:Basis})
\begin{equation}\label{feasibleB}
  x_B=A\columnIndex{B}\inv q\enspace, \ \ \ x_N=0\enspace.
\end{equation}
A basis $B$ is called \emph{complementary} if $|\{i, i+n\} \cap B|= 1$ for all $i=1,\ldots, n$, and
\emph{feasible} if the associated basic solution satisfies the nonnegativity constraint in (\ref{eq:Basis}),
i.e. $A\columnIndex{B}^{-1}q\geq 0$. Every complementary feasible basis defines a solution of the
LCP~\eqref{eq:Def1LCP}, by setting $(w^T,z^T)=x^T$. In the parametric case, each basis is feasible for a set
of parameters, which leads to the notion of a critical region.
\begin{defi}
  The \emph{critical region} $\ERRE_B$ of a complementary basis $B$ is defined as the set of all parameter
  values for which $B$ is feasible, i.e.,
  \begin{equation}
    \ERRE_B:=\{\theta\in \real^d\,\vert\, A\columnIndex{B}\inv (q+Q\theta)\geq0\}\enspace.
  \end{equation}
  A complementary basis $B$ is called \emph{feasible} for the pLCP~\eqref{eq:pLCP_equations} if $\ERRE_B$ is
  nonempty.
\end{defi}

By definition critical regions are convex polyhedra contained in the set of feasible parameters
$\Theta_f$. Each feasible complementary basis $B$ defines a solution of the pLCP for each $\theta\in\ERRE_B$
as $\begin{bmatrix}w(\theta)\\z(\theta)\end{bmatrix}_B=A\columnIndex{B}\inv(q+Q\theta)$ and
$\begin{bmatrix}w(\theta)\\z(\theta)\end{bmatrix}_N=0$, which is an affine function in $\ERRE_B$. As a result,
if $\Theta_f$ can be partitioned into a set of critical regions whose interiors are disjoint, then we have
immediately a piecewise affine solution of pLCP~\eqref{eq:pLCP_equations} defined over these critical regions.

In this paper we define a set of conditions under which such a partitioning can be achieved and introduce an
efficient algorithm for this class of problems. The algorithm is based on the tracing of a graph whose nodes
are the full-dimensional critical regions and whose edges are the pairs of adjacent regions (having a
$(d-1)$-dimensional intersection).

\begin{defi} \label{def:cr_adjacency} Two critical regions $\ERRE_1, \ERRE_2$ are called \emph{adjacent} if
  their intersection $\ERRE_1\cap \ERRE_2$ is of dimension $d-1$.
\end{defi}
\begin{defi}\label{def:cr_adjacency_graph}
  Let $V$ be the set of complementary bases $B$ of pLCP~\eqref{eq:pLCP_equations} such that $\ERRE_B$ is
  full-dimensional and let $E$ be the set of pairs of bases in $V$ whose critical regions are adjacent. The
  graph $\GI:=(V,E)$ is called the \emph{critical region graph} of the pLCP~\eqref{eq:pLCP_equations}.
\end{defi}

The proposed algorithm enumerates all full-dimensional critical regions by tracing the above graph. This
tracing requires that we are able to enumerate all neighbors of a given complementary basis.  The following
section discusses the properties of this graph and investigates restrictions on matrices $M$ under which the
neighbor search can be done efficiently.

\section{Well behaving matrix classes for parametric LCPs}\label{section:sufficient_matrices}
The goal of solving a parametric LCP is to compute functions $w(\cdot)$ and $z(\cdot)$ that
satisfy~\eqref{eq:pLCP_equations} for all feasible values of the parameter $\theta\in\Theta_f$. As discussed
in the introduction, linear complementarity problems include a very large set of difficult optimization
problems and so we cannot hope for a solution in the general case. In this section, we identify classes of
LCPs that are `well-behaving', or that have properties which guarantee that the algorithm given in
Section~\ref{section:Algo_GP} will find a solution.

The two key properties that will be needed are convexity of the feasible set $\Theta_f$ and the existence of a
``canonical'' single-valued mapping from parameters to critical regions. The latter essentially means that the
relative interiors of critical regions do not intersect. In this section we will formalize these notions and
discuss a well-known matrix class that has the appropriate properties when the affine subspace $S
\subset\real^n$ of all possible right hand sides is the whole space $\real^n$. In Section~\ref{sec:critical
  domains} we will then generalize this and give conditions such that these properties still hold when the
right hand side is restricted to lie in some lower-dimensional affine subspace.\footnote{Throughout the paper,
  we use the same notation regarding matrix classes as in~\cite{cottle} and we use the properties of each
  class proved there. At the end of the paper we append an auxiliary section, where the relevant definitions
  and theorems are mentioned.}

\subsection{Complementary cones}
We begin by describing the set of right hand sides $q$ in \eqref{eq:pLCP_equations} that are feasible for a
given set of active constraints.

For any index $i\in \{1,\dots , 2n\}$ we denote with $\bar i$ the \emph{complementary index} of $i$,
i.e. $\bar i = (i+n) \mod 2n$. For a set $I\subseteq \{1,\dots , 2n\}$, the set $\bar I $ is defined as the
set of all complementary indices of elements in $I$.  A set $J\subset\{1,\dots,2n\}$ is called
\emph{complementary\/} if $i\in J $ implies $\bar i\not\in J$.

\begin{defi}
  For any complementary set $J$, the cone $\CI(J):=\cone(A\columnIndex{J})$ is called a \emph{complementary
    cone\/} (relative to $M$), where $\cone(T)$ denotes the cone of all nonnegative combinations of the
  columns of a matrix $T$.
\end{defi}

If $B$ is a complementary basis, then the complementary cone $\CI(B)$ is full-dimensional, and conversely if
the complementary cone $\CI(J)$ is full-dimensional then the submatrix $A\columnIndex{J}$ has full rank,
i.e. $J$ is a complementary basis. For a complementary basis $B$, we have
\begin{align}\label{eq:compl_cone}
  \CI(B)=\{y\in\real^n\,\vert\, A\columnIndex{B}\inv y\geq 0\}\enspace.
\end{align}
In the remainder of the paper we will denote by $\beta$ the matrix $ A\columnIndex{B}\inv $, where $B$ is the
considered basis. Therefore we will write $\CI(B)=\{y\in\real^n\,\vert\, \beta y\geq 0\}$.

One can see that for a given basis $B$, the cone $\CI(B)$ is the set of all right hand sides that are feasible
for LCP~\eqref{eq:Def1LCP}. We are interested in LCPs that have complementary cones with disjoint interiors
and so we introduce the class of sufficient matrices, which has this property.
\begin{defi}
  A matrix $M\in \real^{n\times n}$ is called \emph{column sufficient} if it satisfies the implication
  \begin{equation} [z_i(Mz)_i\leq0 \mbox{ for all }i]\; \Longrightarrow \;[z_i(Mz)_i=0\mbox{ for all
    }i]\enspace.
  \end{equation}
  The matrix $M$ is called \emph{row sufficient} if its transpose is column sufficient. If $M$ is both column
  and row sufficient, then it is called \emph{sufficient}.
\end{defi}
\begin{rem}
  We note that both positive semidefinite (abbreviated by PSD) and $\PI$-matrices are sufficient. For a given
  matrix $M$ it is possible to test in finite time whether it is sufficient, although no polynomial time test
  is currently known.
\end{rem}

The class of LCPs with sufficient matrices has been studied extensively, partly because this class appears to
capture all critical structures for LCPs to behave nicely.  In particular, this class admits many fruitful
results ranging from combinatorial algorithms and duality \cite{ft-lcom-92,fnt-eptlc-98} to the efficient
solvability by interior-point methods \cite{kmny-uaipa-91}.  We will see that this class is ideal also for the
investigation of parametric LCPs.  We start with a key fact.

\begin{prop} [{\cite[Theorem~6.6.6]{cottle}}] \label{prop:disjointCC} If $M$ is a sufficient matrix, then the
  relative interiors of any two distinct complementary cones are disjoint.
\end{prop}

The union of all complementary cones forms a set known as the \emph{complementary range} $\KAPPA(M)$. The
complementary range is equal to the set of all right hand sides of the LCP for which a feasible solution
exists \cite{cottle}
\begin{equation}\label{eq:compl_range}
  \KAPPA(M):=\{q\,\vert\, \text{the LCP~\eqref{eq:Def1LCP} with matrix $M$ and right hand side $q$ is feasible}\}\enspace.
\end{equation}

\begin{prop}\label{prop:convex}
  If $M$ is a sufficient matrix, then the complementary range $\KAPPA(M)$ is a convex polyhedral cone
  $\KAPPA(M)=\cone([I\quad-M])$.
\end{prop}
\begin{proof}
  The statement follows from the fact that sufficient matrices are in $\QU_0$, see Theorem \ref{thm:rowsuff
    P0Q0}.
\end{proof}

\begin{rem}
  Throughout the paper we will draw upon the properties of two matrix classes extensively. The first class is
  the $\QU_0$-matrices, whose complementary range $\KAPPA(M)$ is a convex cone and the second is the
  \emph{fully semi-monotone} matrices, denoted by $\fullySemimo$ which have complementary cones that are all
  disjoint in their interiors.  The class of sufficient matrices is contained in $\QU_0\cap\fullySemimo$ and
  is perhaps the largest known subclass defined by a simple set of conditions, which is why sufficiency is
  assumed for the majority of the results in this paper. It should be noted, however, that many of the results
  hold under slightly relaxed assumptions.
\end{rem}

We will study now the adjacency relationship of complementary cones for the case of sufficient
matrices. Specifically, since our goal is to compute the critical region graph $\GI$, finding all neighbors of
any given region is a crucial issue.  We first look at the neighbors of a complementary cone that determine
possible candidates for the neighbors for a critical region.

\begin{defi}
  Two complementary bases $B_1$ and $B_2$ are called \emph{adjacent} if their cones $\CI(B_1)$ and $\CI(B_2)$
  are \emph{adjacent}, that is, the dimension of $\CI(B_1)\cap \CI(B_2)$ is $n-1$.
\end{defi}

The following lemma is important in narrowing down the candidates of the neighbor search.

\begin{lem} \label{thm:adjacency_CSU_necessary} If $M\in\real^{n \times n}$ is a sufficient matrix and $B_1$
  and $B_2$ are adjacent complementary bases, then $|B_1\cap B_2 | \geq n-2$.
\end{lem}
\begin{proof}
  By the definition of adjacency the intersection $\CI(B_1)\cap \CI(B_2)$ has dimension $n-1$ and therefore
  there exists a $q\in \KAPPA(M)$ that lies in the relative interior of a facet of both complementary cones,
  which means that both basic solutions $(w_1,z_1),\,(w_2,z_2)$ have exactly $n-1$ strictly positive
  components. Recall that basic solutions can be stated as:
  \begin{align*}
    \begin{bmatrix}
      w_1\\
      z_1
    \end{bmatrix}_{B_1} &=A\columnIndex{B_1}\inv q\enspace, &
    \begin{bmatrix}
      w_2\\
      z_2
    \end{bmatrix}_{B_2} &=A\columnIndex{B_2}\inv q\enspace.
  \end{align*}
  Let $J$ be a subset of $B_1$ such that $(w_1^T,z_1^T)_{J}>0$ and $\vert J\vert =n-1$. By Theorem
  \ref{thm:characterisation_CSu}, we have $(w_2^T,z_2^T)_{\bar J}=0$. Since $(w_2^T,z_2^T)_{B_2}$ has exactly
  one zero component, at least $n-2$ elements of $\bar J$ are not in $B_2$ and therefore their complements
  are.  This shows $|B_1\cap B_2 | \geq n-2$.
\end{proof}

\begin{rem}
  It can be shown for $\PI$-matrices that two bases are adjacent if and only if they differ by exactly one
  element.  This implies that the set of all complementary cones for a $\PI$-matrix LCP together with their
  faces forms a polyhedral complex.  Unfortunately, this polyhedral complex property is not satisfied in
  general for sufficient matrices, nor in fact for the proper subclass of PSD matrices.  More precisely, the
  intersection of two critical regions may not be a common face, see \cite{Jones_Morari}.
\end{rem}

\begin{lem}\label{thm:adjacency_CSU_necessary_2}
  Let $M\in\real^{n\times n}$ be sufficient and $B$ be a complementary basis.  If $B'=B\backslash\{i\}\cup
  \{\bar i\}$ is a basis then $\CI(B)$ and $\CI(B')$ intersect in their common facet
  $\CI(B\backslash\{i\})$. Moreover no other full-dimensional complementary cones intersect the relative
  interior of $\CI(B\backslash\{i\})$, i.e. $\CI(B')$ is the unique complementary cone adjacent to $\CI(B)$
  along this facet.
\end{lem}
\begin{proof}
  Since $B\backslash\{i\}$ is a subset of $B$ and $B'$, $\CI(B\backslash\{i\})$ is a common facet of $\CI(B)$
  and $\CI(B')$. The second statement follows directly from the fact that the interior of any other
  complementary cone can intersect neither $ \CI(B)$ nor $\CI(B')$, since $M$ is sufficient.
\end{proof}

\begin{rem}\label{rem:bounded_adj_cones}
  Lemmas~\ref{thm:adjacency_CSU_necessary} and~\ref{thm:adjacency_CSU_necessary_2} imply that a complementary
  basis has at most $n+\frac{n^2-n}{2}$ adjacent complementary bases.
\end{rem}
Given a complementary basis $B$ one can see that replacing any index $i\in B$ with its complement $\bar i$
preserves complementarity, i.e. $B\backslash\{i\}\cup\{\bar i\}$ is still a complementary set. This operation
is called a \emph{diagonal pivot}. If we substitute two different indices $i,j\in B$ with their complements,
then the operation is called an \emph{exchange pivot}.  Lemma~\ref{thm:adjacency_CSU_necessary} ensures that
for a given basis $B$ we can reach all adjacent bases by a single diagonal pivot or by a single exchange pivot
operation. However, for some $i\in B$ the set $B\cup\{\bar i\}\backslash\{i\}$ may not be a basis, or for some
pair $(i,j)\in B$ the basis $B\cup\{\bar i, \bar j\}\backslash\{i,j\}$ may not be adjacent to $B$. Therefore,
in order to determine whether a set given by a diagonal or an exchange pivot is in fact an adjacent feasible
basis we need a further condition. Such a condition can be easily derived from the \emph{dictionary} of the
basis $B$.

\begin{defi}
  Given a complementary basis $B$ and its complement $N$ the matrix $D:=-A\columnIndex{B}\inv
  A\columnIndex{N}\in \real^{B\times N}$ is called the \emph{dictionary} of $B$.
\end{defi}
We begin by examining the diagonal pivot, for which a well-known adjacency condition can be derived.

\begin{fact}
  If $B$ is a complementary basis, then for any $i\in B$, the set $B\backslash\{i\}\cup \{\bar i\}$ is a basis
  if and only if $D_{i,\bar i}\neq0$.
\end{fact}

We now consider the exchange pivot and derive necessary and sufficient conditions for adjacency, which are
again based on examining elements of the dictionary.

\begin{prop}\label{prop:adjacencyCC}
  Let $M\in\real^{n\times n}$ be a sufficient matrix, $B$ be a complementary basis and $D$ be its
  dictionary. Consider the complementary basis $B^\prime=B\backslash\{i,j\}\cup\{\bar i,\bar j\}$, where
  $i,j\in B$ are distinct. The following condition holds:
  \begin{eqnarray}
    dim(\CI(B\backslash\{i\})\cap \CI(B^\prime))=n-1 \iff D_{i \bar i}=0 \text{ and } D_{j \bar i}<0.
  \end{eqnarray}
\end{prop}

\begin{proof}
  Define $\alpha_k:=-D_{k\bar i}$, then the following holds:
  \begin{equation}\label{eq:coeff_adjacencyCC}
    A\columnIndex{\bar i}=\sum_{k\in B\backslash \{i\}}\alpha_k A\columnIndex{k},
  \end{equation}
  Let $j\in B\backslash\{i\}$, since $\alpha_j\neq 0$ (we have assumed $B^\prime$ to be a basis) we can
  rewrite (\ref{eq:coeff_adjacencyCC}) as
  \begin{equation}\label{eq:coeff_adjacencyCC_respect_j}
    A\columnIndex{j}=\frac{1}{\alpha_j}\left(A\columnIndex{\bar i}-\sum_{k\in B\backslash\{i,j\}} \alpha_k A\columnIndex{k}\right)\enspace.
  \end{equation}
  Let us consider
  \begin{equation}\label{eq:def_of_q_lambda}
    q(\lambda)=\sum_{k\in B\backslash\{i\}} \lambda_k A\columnIndex{k}\enspace,
  \end{equation}
  which lies in the relative interior of $\CI(B\backslash\{i\})$ if and only if $\lambda_k>0$ for all $k$.

  We can express $q(\lambda)$ in following way by substituting (\ref{eq:coeff_adjacencyCC_respect_j}) in
  (\ref{eq:def_of_q_lambda}):
  \begin{eqnarray}
    q(\lambda)=\sum_{k\in B\backslash\{i,j\}}\lambda_k A\columnIndex{k}+\lambda_j(1/\alpha_j A\columnIndex{\bar i}-\sum_{k\in B\backslash\{i,j\}} \alpha_k/\alpha_j A\columnIndex{k})\\ 
    =\sum_{k\in B\backslash\{i,j\}}(\lambda_k-\alpha_k/\alpha_j ) A\columnIndex{k}+\lambda_j/\alpha_j A\columnIndex{\bar i}\label{eq:q_lambda_lin_comb_B2}
  \end{eqnarray}
  Sufficiency: if $\alpha_j>0$ then there exists a $q(\lambda)$ that lies in the relative interior of both facets $\CI(B\backslash\{i\})$ and $\CI(B^\prime\backslash\{\bar j\})$.\\
  Necessity: since $B^\prime$ is a basis the unique way to express $q(\lambda)$ as a linear combination of the
  vector indexed by $B^\prime$ is (\ref{eq:q_lambda_lin_comb_B2}). If $\alpha_j<0$ any $q (\lambda)\in
  \relINT(\CI(B\backslash\{i\}))$ can not lie in $\CI(B^\prime)$. The case $\alpha_j=0$ is impossible since we
  have assumed $B^\prime$ to be a basis.
\end{proof}
Corollary~\ref{cor:No_adj_cone} follows directly from the proposition above and allows the detection of the
boundaries of the complementary range.
\begin{cor}\label{cor:No_adj_cone}
  Let $M\in \real^{n\times n}$ be sufficient, $B$ be a complementary basis and denote $A\columnIndex{B}\inv$
  as $\beta$. Consider the facet $\CI(B\backslash \{i\})$, for any $i\in B$. The hyperplane
  $\aff(\CI(B\backslash \{i\}))=\{y\in\real^n\,\vert\, \beta\rowIndex{i}y=0\}$ defines a facet of the
  complementary range $\KAPPA(M)$ if and only if $D_{i \bar i}=0$ and $D\rowIndex{i}\geq \mbzero$.
\end{cor}
\begin{rem}
  Lemma~\ref{thm:adjacency_CSU_necessary} and Proposition~\ref{prop:adjacencyCC} are valid also for column
  sufficient matrices (i.e. it is not necessary that $M$ is row
  sufficient). Lemma~\ref{thm:adjacency_CSU_necessary_2} holds for all fully semimonotone matrices.
\end{rem}

\subsection{Critical Domains}\label{sec:critical domains}
We study now the parametric case where the right hand side of LCP~\eqref{eq:Def1LCP} is restricted to lie
within some affine subspace $S := \{q+Q\theta\,\vert\, \theta\in\real^d\}$. We will see, under some
assumptions on $S$, that the properties of the complementary cones discussed in the previous section still
hold in this case.

\begin{defi}
  If $B$ is a complementary basis, then the \emph{critical domain} $S_B$ is the intersection of the affine
  subspace $S$ with the complementary cone $\CI(B)$
  \begin{equation}\label{eq:def_CriticalDomain}
    S_B:=\CI(B)\cap S=\{y\,\vert\, A\columnIndex{B}\inv y\geq 0,\, y=Q\theta+q ,\,\theta\in\real^d \}\enspace.
  \end{equation}
\end{defi}

Since we have assumed $Q$ to be full column rank, the parametrisation $q+Q\theta$ is an invertible function
and it is not hard to see that a critical domain is therefore the image of a critical region, i.e. $S_B=Q
\ERRE_B+q$.  Since the parametrisation is a bijection, $S_B$ and $\ERRE_B$ have the same combinatorial
structure for any complementary basis $B$. In particular, we have:
  
\begin{rem} \label{rem:affinebijection} The inequality $\beta\rowIndex{i} y \geq 0$ is redundant in $S_B$ if
  and only if $\beta\rowIndex{i} (Q \theta + q) \geq 0$ is redundant in $\ERRE_B$, where
  $\beta=A\columnIndex{B}\inv$.
\end{rem}

We now define a key assumption, which will allow the extension of the properties of complementary cones to
critical domains.
\begin{defi}
  The affine subspace $S$ is said to lie in \emph{general position} if for every complementary basis $B$ the
  following condition holds
  \begin{equation}
    S \text{ intersects }\CI(B)\Rightarrow S \text{ intersects }\operatorname{int}(\CI(B))\enspace.
  \end{equation}
\end{defi}

If a critical domain has dimension $d=\dim(S)$, we simply say that it is \emph{full-dimensional}.  By the
definition above, we have:

\begin{rem}
  If $S$ lies in a general position, then every critical domain is either full-dimensional or empty.
\end{rem}
\begin{prop}\label{prop:disjointCD}
  If $M$ is sufficient and $S$ lies in general position, then the relative interiors of critical domains
  $S_{B_1}$ and $S_{B_2}$ are disjoint for any two distinct complementary bases $B_1$ and $B_2$.
\end{prop}
\begin{proof}
  The statement is a direct consequence of Proposition~\ref{prop:disjointCC}, i.e. $\INT(\CI(B_1))$ and
  $\INT(\CI(B_2))$ are disjoint, and from $\relINT(S_{B_i})\subseteq \INT(\CI(B_i))$ for $i=1,2$.
\end{proof}
We denote the set of the feasible points of $S$ by $S_f$.  By (\ref{eq:compl_range}) it follows that
$S_f=S\cap \KAPPA(M)$.
\begin{cor}\label{cor:Sf is convex}
  If $M\in\real^{n\times n}$ is sufficient then $S_f$ is a convex polyhedron.
\end{cor}
\begin{proof}
  The statement is a direct consequence of Proposition \ref{prop:convex}.
\end{proof}

If $M$ is a sufficient matrix and $S$ is in general position, then Proposition~\ref{prop:disjointCD} and
Corollary~\ref{cor:Sf is convex} ensure that the set $K$ of nonempty critical domains defines a
\emph{polyhedral decomposition of $S_f$\/} in the sense that
\begin{itemize}
\item each member $P$ of $K$ is a convex polyhedron,
\item $\cup_{P \in K} P = S_f$,
\item $\dim P = d$ for all $P\in K$, and
\item $\dim(P\cap P') \le d-1$ for any two distinct members $P$ and $P'$ of $K$.
\end{itemize}
It is important to note that the set $K$ may not induce a polyhedral complex, i.e. the intersection of two
critical domains may not be a common face.  Nevertheless, because $S_f$ is convex, we can define a graph
structure of the decomposition which is connected.

\begin{defi}\label{def:graph_of_CD} 
  Let $V$ be the set of complementary bases $B$ of pLCP~\eqref{eq:pLCP_equations} such that $S_B$ is
  full-dimensional and $E$ consist of edges connecting each pair of bases in $V$ whose critical domains are
  adjacent.  The graph $\GI:=(V,E)$ is called the \emph{critical domain graph} of the
  pLCP~\eqref{eq:pLCP_equations}.
\end{defi}

As stated above, each critical domain $S_B$ is the image of a critical region $\ERRE_B$ under the affine map
$\theta\mapsto q+Q\theta$, and a similar statement can be make for the feasible sets $S_f= q+Q\Theta_f$. Since
for each complementary basis $B$ the critical domain $S_B$ and the critical region $\ERRE_B$ have the same
combinatorial structure, the critical domain graph $\GI=(V,E)$ also defines the graph of critical regions and
vice versa. In the discussion of the algorithm we will mostly consider only critical domains.

\begin{cor}\label{cor:adj_graph_connected} If $M$ is a sufficient matrix and $S$ lies in
  general position, then the graph of critical domains $\GI$ is connected.
\end{cor}
\begin{proof}
  The statement follows directly from the convexity of $S_f=\KAPPA(M)\cap S$, which implies that between every
  pair of critical domains there exists a path in the graph of critical domains.
\end{proof}

In the previous section we have seen that for the case of sufficient matrices we can reach all adjacent cones
from any complementary cone with a single diagonal or a single exchange pivot operation. Assuming general
position of $S$, this useful property also holds for critical domains.
\begin{prop}\label{prop:adjcencyCD_CC}
  If $M$ is sufficient and $S$ lies in general position, then for any two complementary bases $B_1$ and $B_2$
  that have nonempty critical domains the following holds: If $S_{B_1}$ and $S_{B_2}$ are adjacent in $S$ then
  $\CI(B_1)$ and $\CI(B_2)$ are adjacent cones.
\end{prop}
\begin{proof}
  If $S_{B_1}$ and $S_{B_2}$ are adjacent critical domains, then their intersection is contained in
  $\CI(B_1)\cap\CI(B_2)$. If $\CI(B_1)$ and $\CI(B_2)$ are not adjacent cones then there exists a
  complementary cone $C$ adjacent to $\CI(B_1)$ that contains $\CI(B_1)\cap\CI(B_2)$. In this case $C\cap S$
  would be adjacent to $S_{B_1}$ and would overlap with the relative interior of $S_{B_2}$, which is a
  contradiction of Proposition~\ref{prop:disjointCD}.
\end{proof}
Given a complementary feasible basis $B$, Proposition~\ref{prop:adjcencyCD_CC} ensures that all the critical
domains adjacent to $S_B$ can be reached by exploring complementary bases adjacent to $B$.

\section{Description of the generic algorithm}\label{section:Algo_GP}
Now we are able to present an algorithm that enumerates all complementary bases whose critical domains define
a polyhedral partition of $S_f$ with the following two sets of assumptions:
\begin{assu}[Regularity] \label{assu:regularity} The matrix $M$ is sufficient and the matrix of the
  parametrisation $Q\in \real^{n\times d}$ has full column rank.
\end{assu}

\begin{assu}[General Position] \label{assu:generalposition} The affine subspace $S:=\{q+Q\theta\,\vert\,
  \theta\in\real^d\}$ lies in general position with respect to the complementary cones relative to $M$.
\end{assu}

Assumption \ref{assu:regularity} is essential for our algorithm to work, whereas Assumption
\ref{assu:generalposition} will be relaxed in the next section where an extension of the algorithm simulating
general position for any given affine subspace via a symbolic perturbation is presented.

The proposed algorithm given in Algorithm~\ref{alg:graph_search} is based on a standard graph search
procedure.  It assumes a given function $\textit{neighbors}(B)$, which returns all bases whose critical
domains are adjacent to that of a given basis $B$.  The validity follows immediately from the connectivity of
the critical domain graph, Corollary \ref{cor:adj_graph_connected}.  As input it takes a matrix $M$ and an
affine subspace $S$ that satisfy the above assumptions, as well as an initial feasible complementary basis
$B_0$ such that $S_{B_0}$ is full-dimensional. The basis $B_0$ is flagged as ``unexplored'' and added to the
set of discovered bases $\BI$. In each iteration of the algorithm an unexplored basis $B$ is selected from
$\BI$, marked as ``explored'' and all bases that have adjacent critical domains are enumerated and added to
$\BI$, marking the new bases as ``unexplored''. Once all bases in $\BI$ have been explored, then we have found
all bases with full-dimensional critical domains.

\begin{algorithm}[ht]
  \caption{Enumerate all critical domains by graph search}\label{alg:graph_search}
  \inputbox{A feasible basis $B_0$ with $dim(S_{B_0})=d$, a sufficient matrix $M\in\real^{n \times n}$ and an affine subspace $S$ that lies in general position.} \\[1ex]
  \outputbox{The critical domain graph $\GI=(\BI,E)$.}
  \begin{algorithmic}[1]
    \STATE \textbf{Initialise} the set of nodes $\BI:=\{B_0\}$ and edges $E:=\emptyset$. %
    \STATE Flag $B_0$ as ``unexplored''.  \WHILE{there exists an unexplored basis $B$ in $\BI$}%
    \STATE Flag $B$ as ``explored'' \STATE $\BI_{\text{new}}:=\textit{neighbors}(B)$\COMMENT
    {\parbox[t]{3in}{$\BI_{\text{new}}:=\textit{neighbors}^\epsilon(B)$ if $S$ does not lie in general
        position}}\label{line:exploration}%
    \STATE Flag each $B^\prime \in \BI_\text{new} \backslash \BI$ as
    ``unexplored''\label{line:explored_condition}%
    \STATE $\BI:=\BI\cup \BI_{\text{new}}$ \STATE $E:=E\cup (B,B')$ for each $B'\in\BI_{\text{new}}$
    \ENDWHILE
    \STATE \textbf{return} $\GI=(\BI,E)$
  \end{algorithmic}
\end{algorithm}

The remainder of this section describes how the results of the previous sections can be exploited to
efficiently enumerate all adjacent critical domains of a given basis, i.e., how the function
$\textit{neighbors}(\cdot)$ can be properly implemented. The following section will then detail how the method
can be extended so that the general position assumption can be relaxed.

\subsection{Neighborhood computation of a critical domain}\label{subsection:Neighborhood_comput_GP_case}
This section details a computational method that enumerates all bases that define adjacent critical domains of
a given basis, i.e. how the basis is ``explored'', under both Assumption \ref{assu:regularity} and Assumption
\ref{assu:generalposition}.

The function \emph{neighbors} is given as Algorithm~\ref{alg:neighbors}.  Let $B$ be a basis whose critical
domain is full-dimensional.  By Proposition~\ref{prop:adjcencyCD_CC}, each adjacent critical domain must have
a $(d-1)$ dimensional intersection with a facet of $S_B$. We begin therefore by first computing all facets of
$S_B$ and then by determining the critical domains that intersect each one.

Given a complementary feasible basis $B$ we determine which facets of $\CI(B)$ define the facets of $S_B$ by
removing the redundant inequalities of $S_B=\{y\in\real^n\,\vert\, \beta y\geq 0, \; y=Q\theta+q \}$, where
$\beta:= A\columnIndex{B}\inv$. The hyperplane $h_i=\{y\,\vert\, \beta\rowIndex{i}y=0 \}, i\in B$ intersected
with $S_B$ is a facet of $S_B$ if there exists a $y^*\in S_B$ such that $\beta\rowIndex{j} y^*>0 $ for all
$j\in B\backslash\{i\}$ and $\beta\rowIndex{i}y^*=0$.  This fact relies on the general position assumption,
Assumption \ref{assu:generalposition}.  Therefore $h_i\cap S_B$ is a facet of $S_B$ if and only if the
following LP:
\begin{equation}\label{eq:LP_facet}
  \begin{array}{rccc}
    t^*=&\text{max }& t &\\ 
    &\text{s.t.}& -\beta\rowIndex{j} Q\theta+t\leq \beta\rowIndex{j}q,&\forall j\in B\backslash\{i\} \\
    & &-\beta\rowIndex{i}Q\theta=\beta\rowIndex{i}q&
  \end{array}
\end{equation}
has an optimal value $t^*>0$ strictly positive.

\begin{algorithm}[ht]
  \caption{Function $\textit{neighbors}(B)$: returns all bases whose critical domains are adjacent to
    $S_B$.}\label{alg:neighbors}
  \inputbox{A complementary basis $B$, the matrix $M$ and the affine subspace $S$. $M$ is assumed
    to be sufficient and $S$ to lie in general position.}\\[1ex]
  \outputbox{The set $\BI$ of complementary bases, whose critical domains are adjacent to $S_B$.}
  \begin{algorithmic}[1]
    \STATE $\BI:=\emptyset$ \STATE $\beta:=A\columnIndex{B}\inv$ \STATE $D:=-A\columnIndex{B}\inv
    A\columnIndex{N}\in \real^{B\times N}$\COMMENT{where $N=\{1,\dots,2n\}\backslash B$} \FOR{each $i\in B$}
    \IF[by solving LP (\ref{eq:LP_facet})] {$\beta\rowIndex{i} y\geq 0$ is non-redundant in
      $S_B$ \label{line:alg:redcheck} }\label{line:alg:redundancy} \IF{$D_{i\bar i}>0$} \STATE add
    $B\backslash\{i\}\cup\{\bar i\}$ to $\BI$ \ELSE \FOR {each $j\in B$ with $D_{i\bar j}<0$}
    \STATE{$B'':=B\backslash\{i,j\}\cup\{\bar i,\bar j\}$ } \IF [by solving LP
    (\ref{eq:adjacencyCheck_exchange_pivot_non_degenerate})]{$dim(S_B \cap
      S_{B''})=d-1$}\label{line:alg:adjacency} \STATE add $B''$ to $\BI$
    \ENDIF
    \ENDFOR
    \ENDIF
    \ENDIF
    \ENDFOR
    \RETURN{$\BI$}
  \end{algorithmic}
\end{algorithm}

By solving LP~\eqref{eq:LP_facet} for each $i\in B$ we can determine if $\CI(B\backslash\{i\})$ defines a
facet of $S_B$ or not; see Line~\ref{line:alg:redcheck} of Algorithm~\ref{alg:neighbors}. If it does, then the
goal is to determine which bases, if any, have critical domains that intersect this facet. From the previous
section, we saw that there are three possible cases:

\paragraph{Diagonal pivot.} If {$D_{i\bar i}>0$} then the cone $\CI(B')$ defined by
$B':=B\backslash\{i\}\cup\{\bar i\}$ is the unique complementary cone adjacent to $\CI(B)$ along the facet
$\CI(B)\cap h_i$, due to Lemma~\ref{thm:adjacency_CSU_necessary_2}, and therefore $S_{B'}$ is the unique
critical domain adjacent to $S_B$ along the facet $S_B\cap h_i$. Since $S_{B'}$ is nonempty (as $S_B\cap h_i$
is included in $S_{B'}$) it is full-dimensional by Assumption \ref{assu:generalposition}.
\paragraph{Boundary of $\KAPPA(M)$.} $\CI(B\backslash\{i\})$ is a facet of the complementary range and
therefore no other complementary cones intersect it. From Corollary~\ref{cor:No_adj_cone}, this is the case
when $D_{i\bar i} = 0$ and $D_i \ge \mbzero$.

\paragraph{Exchange pivot.} By looking at the dictionary $D$ of $B$ all complementary cones adjacent to $C(B)$
that contain the index $\bar i$ can be determined (see Theorem~\ref{prop:adjacencyCC}). However, not all such
cones intersect $S_B$ with dimension $d-1$ and for this reason we need to test for each cone $ \CI(B'')$
adjacent to $C(B)$ whether $ S_{B''}$ is adjacent to $S_B$, i.e. whether the condition $\dim (S_B\cap
S_{B''})=d-1$ holds. Now assume that the basis $B'':=B\backslash\{i,j\}\cup\{\bar i,\bar j\}$, where $j\in
B\backslash\{i\}$, defines such an adjacent complementary cone according to Theorem~\ref{prop:adjacencyCC}. In
order to determine whether $S_{B''}$ is adjacent to $S_B$ we can solve following LP:

\begin{equation}\label{eq:adjacencyCheck_exchange_pivot_non_degenerate}
  \begin{array}{ccccc}
    \text{max} & & t& &\\
    \text{s.t.}& -\beta\rowIndex{k}Q\theta +t&\leq &\beta\rowIndex{k} q&\forall k\in B\backslash\{i\}\\
    & -\beta\rowIndex{k}''Q\theta +t&\leq &\beta\rowIndex{k} q&\forall k\in B''\backslash\{j\}\\
    & -\beta\rowIndex{i}Q\theta& = &\beta\rowIndex{i} q& \\
    & -\beta\rowIndex{j}''Q\theta& = &\beta''\rowIndex{j} q\enspace,& 
  \end{array}
\end{equation}
where $\beta=A\columnIndex{B}\inv$ and $\beta''=A\columnIndex{B''}\inv$. As in the simpler case
(\ref{eq:LP_facet}) above, the intersection $S_B\cap S_{B''}$ has dimension $d-1$ if and only if the optimal
value of (\ref{eq:adjacencyCheck_exchange_pivot_non_degenerate}) is strictly positive. In this case
$S_{B''}=\CI(B'')\cap S$ is nonempty and by the general position assumption, it is also full-dimensional.%

\begin{rem}\label{rem:redundant_equality}
  Since $\CI(B)$ and $\CI(B'')$ are adjacent cones and the hyperplanes $\{y\vert \beta\rowIndex{i}y =0\}$ and
  $\{y\vert \beta\rowIndex{j}'' y=0\}$ defines their shared facet, the two hyperplanes must be
  equivalent. Therefore in (\ref{eq:adjacencyCheck_exchange_pivot_non_degenerate}) one of the equality
  constraints $-\beta\rowIndex{j}''Q\theta = \beta''\rowIndex{j} q$ or $-\beta\rowIndex{i}Q\theta =
  \beta\rowIndex{i} q$ can be removed.
\end{rem}

\begin{figure}
  \centering \subfigure[Two-dimensional slice of a three-dimensional example. The cones $C_2$ and $C_3$ are
  both adjacent to $C_1$ along the same facet. However the affine space $S$ does not intersect $C_2$ and
  therefore $S\cap C_2$ is not adjacent to $S\cap C_1$. This situation arises in the case of an exchange pivot
  and requires that adjacency much be checked by solving an LP.\label{subfig:adjacency_exchange_pivot}]
  {\parbox[t]{0.45\columnwidth}{\includegraphics[width=0.35\columnwidth]{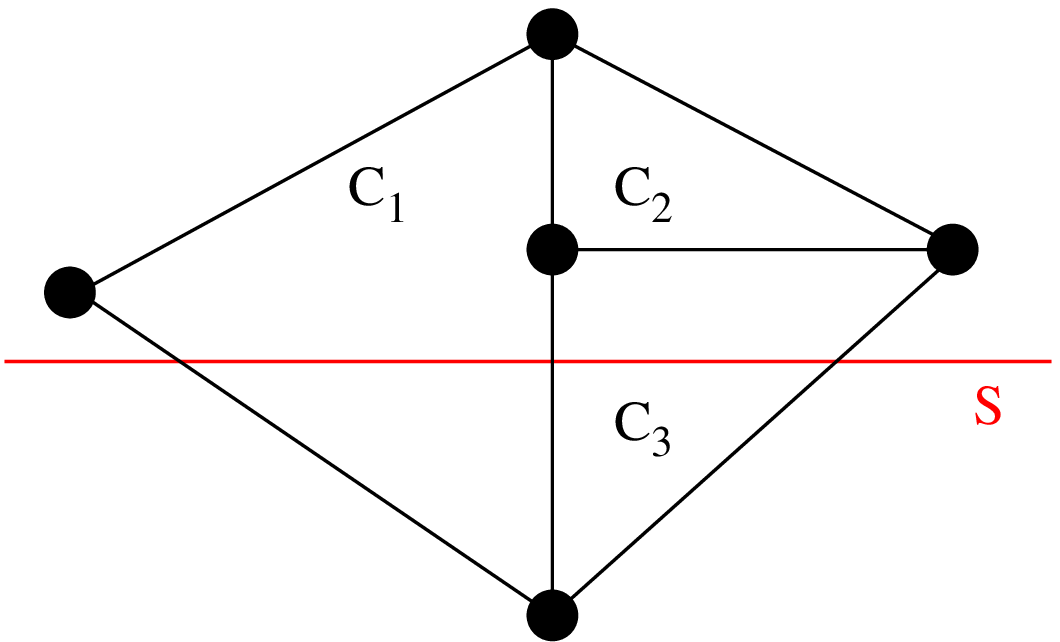}}}%
  \hspace{0.05\columnwidth}%
  \subfigure[Example of diagonal pivot: no adjacency check is needed.\label{subfig:redundant_inequality}]
  {\parbox[t]{0.45\columnwidth}{\includegraphics[width=0.35\columnwidth]{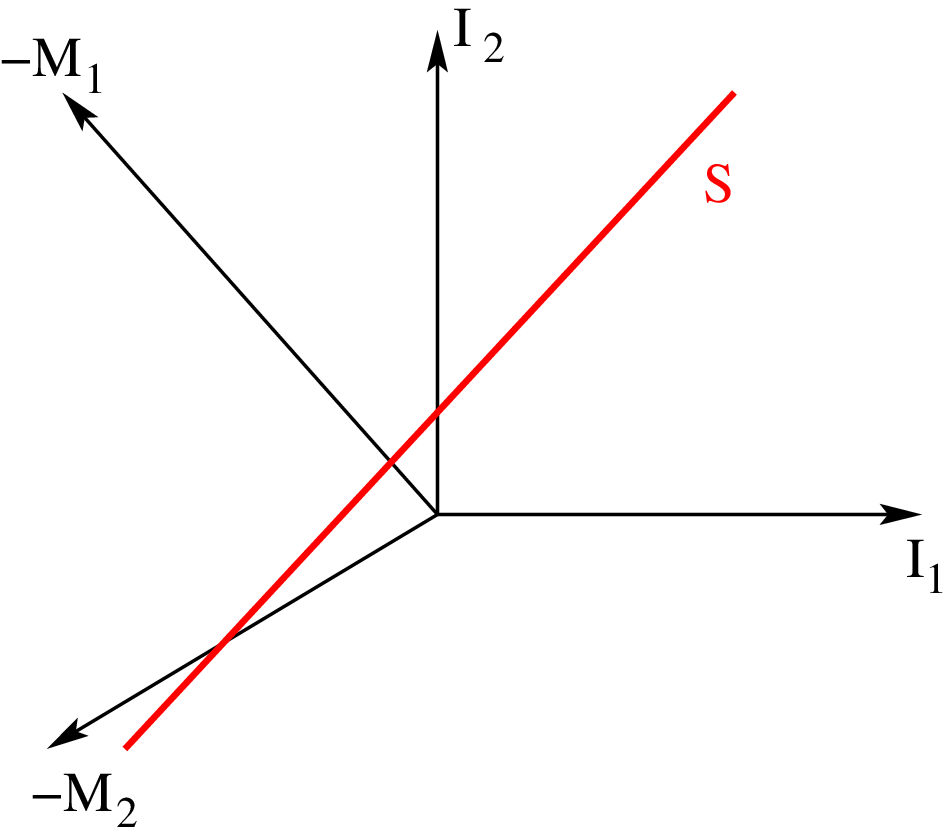}}}
  \caption{Adjacency of critical domains. See Example~\ref{ex:adjacency_pivot} for details.}
\end{figure}

\begin{ex}\label{ex:adjacency_pivot}
  In Figure~\ref{subfig:adjacency_exchange_pivot} a two-dimensional slice of three, three-dimensional cones is
  shown. The cones $C_2$ and $C_3$ are both adjacent to $C_1$ along the same facet. However the affine
  subspace $S$ does not intersect $C_2$ and therefore $S\cap C_2$ is not adjacent to $S\cap C_1$.

  In Figure~\ref{subfig:redundant_inequality} we consider the complementary basis $B_1=\{1,2\}$ and the cone
  $\CI(B_1)=\cone(I)=\{y\geq 0\}$. The goal is to find the critical domain that is adjacent to
  $S_{B_1}=\CI(B_1)\cap S$. The inequality $y_1\geq0$ is not redundant in $S_{B_1}$ and therefore the
  hyperplane $h_1=\{y_1=0\}$ defines a facet of $S_{B_1}$. Since $\cone(-M_1,I_2)\cap h_1$ is equal to
  $\CI(B_1)\cap h_1$, we have that $\cone(-M_1,I_2)\cap S$ is adjacent to $S_{B_1}$. The other inequality
  $y_2\geq0$ is redundant in $S_{B_1}$ and therefore the critical domain $\cone(I_1,-M_2)\cap S$ is not
  adjacent to $S_{B_1}$.
\end{ex}

\section{Extension of the algorithm for $S$ not in general position}\label{section:algo_non_GP}

The previous section presented an algorithm that enumerates all feasible bases and returns the graph of
critical domains.  The algorithm works only under the assumption that the image of the parametrisation $S$
lies in general position; Assumption \ref{assu:generalposition}.  However, this assumption is not realistic
and it is highly desirable to remove it.

In the case of degeneracy (i.e. $S$ is not in general position), Propositions~\ref{prop:disjointCD}
and~\ref{prop:adjcencyCD_CC} are no longer valid, as can be seen in
Example~\ref{ex:no_GP_counterexample}. Therefore, during neighborhood computation it is not sufficient to
explore only the adjacent complementary cones. In order to extend the algorithm to the degenerate case, we
apply a symbolic perturbation technique (the lexicographic perturbation) which has the effect of shifting $S$
into general position.

The next subsection will demonstrate how to handle the perturbation for neighborhood computation, in
particular lines~\ref{line:alg:redundancy} and~\ref{line:alg:adjacency} of Algorithm~\ref{alg:neighbors}. By
using this technique we obtain a graph of critical domains $\GI^\epsilon$ relative to the perturbed affine
subspace $S^\epsilon$, which can differ from the graph of critical domains $\GI$ relative to $S$.  In
particular, some full-dimensional critical domains in $S^\epsilon$ may be non full-dimensional in $S$. We will
see that there exists a subgraph $\GI$ of $\GI^\epsilon$ that is a graph of critical domains relative to $S$
and which can be obtained by postprocessing $\GI^\epsilon$.

\begin{ex}\label{ex:no_GP_counterexample}
  This example demonstrates the effect when a parametric LCP is not in general position. Consider the
  parametric LCP defined by the matrices
$$M=\begin{bmatrix}
  1&-1\\
  1&1
\end{bmatrix}
, \; Q=\begin{bmatrix}
  1\\
  -1
\end{bmatrix} \text{ and }\, q=\begin{bmatrix}
  0\\
  0
\end{bmatrix}\enspace.$$
A figure depicting the complementary cones relative to $M$ and of the affine subspace $S$ is shown in
Figure~\ref{fig:NO_GP_counterexample}. Let $B_1=\{1,2\}$, $B_2=\{2,3\}$, $B_3=\{3,4\}$ and $B_4=\{1,4\}$. For
notational simplicity, we denote by $C_i=\CI(B_i)$ the complementary cones and by $S_i=S\cap C_i$ for
$i=1,\dots ,4$ the critical domains. Clearly $S$ does not lie in general position because it intersects
$C_1,C_3$ and $C_4$ on their boundary but not in their interiors.  Proposition \ref{prop:disjointCD} is
violated because $S_1$ is neither empty nor full-dimensional, and furthermore $S_3$ and $S_4$ are equal and
hence $\relINT(S_3)=\relINT(S_4)$.  Theorem \ref{prop:adjcencyCD_CC} is violated because $S_2$ and $S_4$ are
adjacent, but $C_2$ and $C_4 $ are not.
\end{ex}
\begin{figure}[ht]\label{fig:NO_GP_counterexample}
  \begin{center}
    \includegraphics[width=0.45\columnwidth]{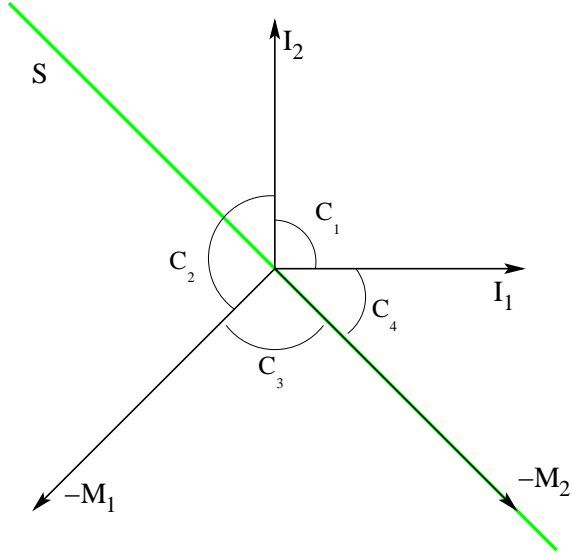}
  \end{center}
  \caption{An affine subspace $S$ that is not in general position. (See Example
    \ref{ex:no_GP_counterexample})}
\end{figure}

\subsection{Lexicographic perturbation}
This section presents a well-known method that permits the perturbation of the image of the parametrisation
$S$ into general position and which can be treated symbolically: the lexicographic perturbation.

We introduce the following notation that will be used for the reminder of the paper.
\begin{defi}
  The vector $\epsLexVector:=(\epsilon, \epsilon^2 , \epsilon^3,\dots,\epsilon^n)^T$ is called the
  \emph{lexicographic perturbation vector\/} and is a function of a positive real number $\epsilon$.
\end{defi}
We denote with $S^\epsilon$ the affine subspace $S$ perturbed by $\epsLexVector$
\begin{equation}
  S^\epsilon:=S+ \epsLexVector=\{y\in\real^n\,\vert\, 
  y=Q\theta+q+\epsLexVector, \theta \in \real^d\}\enspace.
\end{equation}

\begin{thm}\label{thm:lexpert_GP}
  Let $M$ be a sufficient matrix and $S$ be the affine subspace $\{Q\theta+q\,\vert\, \theta\in \real^d\}$ for
  a given matrix $Q$ and a vector $q$.  There exists a $\delta>0$ such that $S^\epsilon:=S+\epsLexVector$ lies
  in general position for each $\epsilon\in(0,\delta)$.
\end{thm}

\begin{rem}
  In the remainder of the paper we will use the standard expression ``property A holds for all
  \emph{sufficiently small} $\epsilon>0$'' rather than the more cumbersome ``there exists $\delta>0$ such that
  property A holds for each $\epsilon\in(0,\delta)$''. Therefore the claim of Theorem~\ref{thm:lexpert_GP} can
  be written as: $S^\epsilon$ lies in general position for all sufficiently small $\epsilon>0$.
\end{rem}

To prove Theorem \ref{thm:lexpert_GP}, we need the following lemma.
\begin{lem}\label{lemma_finitely_many_eps}
  Let $Q\in\real^{n\times d}$, $q\in\real^{n}$ and $S^\epsilon=\{Q\theta+q +\epsLexVector \vert\ \theta \in
  \real^d\}$. For any complementary cone $\CI$, there exist finitely many $\epsilon$ such that $S^\epsilon$
  intersects $\CI$ but not $\INT(\CI)$.
\end{lem}
\begin{proof}
  Let $B$ be a complementary basis. We denote with $h_i$ the hyperplane $h_i:=\{y\vert\beta\rowIndex{i}y=0\}$
  for all $i\in B$. Therefore $\CI(B)\cap h_i$ is a facet of $\CI(B)=\{y\in\real^n\vert \beta y\geq 0\}$. We
  will prove that for any subset $J\subseteq B$ there are finitely many $\epsilon$ such that the following
  condition holds
  \begin{equation}\label{eq:S_B^eps_included_Face}
    \emptyset\neq S_B^\epsilon=\CI(B)\cap S^\epsilon\subseteq  \CI(B)\cap (\bigcap_{i\in J}h_i)\enspace\text{and}\enspace  S_B^\epsilon \not \subseteq h_i \text{ for all }i\not\in J \enspace.
  \end{equation}
  The statement of the lemma will then follow directly, since there are finitely many subsets of $B$.

  Let $J$ be any nonempty subset of $B$. For any $\epsilon$ for which (\ref{eq:S_B^eps_included_Face}) holds,
  $J$ contains the indices of all inequalities of $S^\epsilon$ that are implicit equalities and it holds that
  \begin{equation}\label{eq:inclusion_prop_S_B}
    \beta\rowIndex{J}(Q\theta + q+\epsLexVector)=0\text{ for all }\theta\in\ERRE_B^\epsilon \quad,
  \end{equation}
  where $\ERRE_B^\epsilon=\{ \theta \, | \, Q\theta+q+\epsilon\in S_B^\epsilon \}$.
  
  We can distinguish two cases. The case 1: there exists $i\in J$ such that $\beta\rowIndex{i}Q$ is a zero row
  vector. Since $S_B^\epsilon$ is nonempty, $\beta_i(q+\epsLexVector)=0$ for the
  condition~\eqref{eq:inclusion_prop_S_B} to be valid. This non-trivial polynomial equation holds for at most
  $n$ values of $\epsilon$.
  
  Now consider the case 2: $\beta\rowIndex{J}Q$ has no zero rows.  We will prove that the matrix
  $\beta\rowIndex{J}Q$ does not have full row rank.  Since $S_B^\epsilon$ is nonempty and has dimension
  $d-\rank(\beta\rowIndex{J}Q)<d$, the Chebyshev center problem has an optimal value of zero,
  i.e. $\max\{t~\vert~\beta_i Q\theta -t \geq -\beta_i(q+\epsilon)\, \forall i \}=0$. We consider its dual
  problem
  \begin{equation}\label{eq:dual_chebyshev_proof_Lemma5-4}
    \begin{array}{crcl}
      \text{min} &  (\beta(q+\epsLexVector))^T\,y& & \\
      \text{s.t.}& -(\beta \,Q)^Ty&=&0\\ 
      & \sum y_i&=&1\\ 
      & y &\geq &0 \enspace.
    \end{array}
  \end{equation}
  From strong duality, there exists a non-zero optimal solution $y^*\gneq0$ such that $(\beta Q)^Ty^*=0$ and
  $(\beta(q+\epsilon))^Ty^*=0$. We now claim that all indices of the strictly positive components of $y^*$ are
  contained in $J$. Assume that there is an index $i\not\in J$ with $y^*_i>0$. Then, for any $\theta \in
  \ERRE_B^\epsilon$, $\beta\rowIndex{i}(Q \theta+q+\epsilon)=0$, i.e. $S_B^\epsilon\subseteq h_i$, which
  contradicts the maximality condition in~\eqref{eq:S_B^eps_included_Face}. Therefore, $y^*_J\gneq 0 $ and
  $(\beta Q)^Ty^*_J=0$, i.e. there exists a non-trivial combination of rows of $\beta\rowIndex{J}Q$.

  Let $\{v_1,\dots, v_s\}$ be a set of columns of $Q$ such that $\beta\rowIndex{J}v_1$, $\ldots$,
  $\beta\rowIndex{J}v_s$ form a basis of the column space of $\beta\rowIndex{J}Q$. Since $\beta\rowIndex{J}Q$
  does not have full row rank, $s<\abs{J}$.  If
  $\beta\rowIndex{J}(q+\epsLexVector),\beta\rowIndex{J}v_1,\dots,\beta\rowIndex{J}v_s$ are linearly
  independent then for any $\theta \in \Real^d$ the equation~\eqref{eq:inclusion_prop_S_B} cannot hold.  We
  claim these vectors are linearly dependent for finitely many $\epsilon$.

  First we consider the case $s=\vert J\vert -1$.  Then, these vectors are linearly dependent if and only if
  \begin{equation}
    \det(\beta\rowIndex{J}(q+\epsLexVector),\beta\rowIndex{J}v_1,\dots,\beta\rowIndex{J}v_s)=0\enspace .
  \end{equation}
  This condition is a polynomial equation in $\epsilon$ and holds for finitely many $\epsilon$. Finally, if
  $s<\vert J\vert -1$, we use the same
  argument by adding a proper number of vectors $\bar v_1, \dots, \bar v_{\vert J\vert -1-s}$ such that\\
  $\beta\rowIndex{J}v_1,\dots,\beta\rowIndex{J}v_s,\bar v_1, \dots, \bar v_{\vert J\vert -1-s}$ are $\vert
  J\vert -1$ linearly independent vectors.
\end{proof}

\begin{proof}[Proof of Theorem \ref{thm:lexpert_GP}]
  We can assume without loss of generality that $S$ does not lie in general position. For any complementary
  cone $\CI$ exactly one of the following cases holds:
  \begin{enumerate}
  \item $S$ does not intersect $\CI$ \label{item:disjoint_to_C},
  \item $S$ intersects the interior of $\CI$ \label{item:intersect_int_C},
  \item $S$ intersects the boundary of $\CI$ and $S \cap\INT(\CI)=\emptyset$ ,\label{item:enu_case_only_bound}
    which can be differentiated into two subcases:
    \begin{enumerate}\setcounter{enumii}{0}
    \item $\exists \delta>0$ such that $\CI\cap S^\epsilon=\emptyset$ for all
      $\epsilon\in(0,\delta)$\label{item:enum_case_3a}
    \item For all $\delta>0$ there exists an $\epsilon\in(0,\delta)$ such that $\CI\cap
      S^\epsilon\neq\emptyset$ .\label{item:enum_case_3b}
    \end{enumerate}
  \end{enumerate}
  For each of these cases we need to prove that there exists $\delta>0$ such that either $S^\epsilon$
  intersects $\INT(\CI)$ for all $\epsilon\in(0,\delta)$ or $\CI\cap S^\epsilon=\emptyset$ for all
  $\epsilon\in(0,\delta)$. Clearly, this condition holds for cases~\ref{item:disjoint_to_C},
  \ref{item:intersect_int_C} and~\ref{item:enum_case_3a}.  Therefore, it suffices to prove that for
  case~\ref{item:enum_case_3b} there exists a $\delta>0$ such that $S^\epsilon$ intersects $\INT(\CI)$ for all
  $\epsilon\in(0,\delta)$.

  Let $\CI$ be any complementary cone that satisfies condition~\ref{item:enum_case_3b}. From
  Lemma~\ref{lemma_finitely_many_eps} and by assumption there exists a $\delta>0$ such that $S^\delta$
  intersects the interior of $\CI$ and for any $\epsilon\in(0,\delta)$ either $\INT(\CI)\cap
  S^\epsilon\neq\emptyset$ or $\CI\cap S^\epsilon=\emptyset$.  More precisely, one can select any $\delta>0$
  smaller than the smallest $\epsilon>0$ for which $S^\epsilon$ intersects $\CI$ but not $\INT(\CI)$. Since
  $S^\epsilon$ shifts continuously with $\epsilon$, there exists no $\epsilon\in(0,\delta)$ such that $\CI\cap
  S^\epsilon=\emptyset$.
\end{proof}

\begin{ex}
  Figure~\ref{fig:lexPerturbation} shows two examples in which $S$ does not lie in general position. In the
  first example (Figure~\ref{subfig:lexPerturbation1}) the cones $\CI(\{2,3\})=\cone(I_2, -M_1)$ and
  $\CI(\{1,4\})=\cone(I_1, -M_2)$ contain adjacent critical domains, although they are not adjacent cones.  In
  the second example (Figure~\ref{subfig:lexPerturbation2}) two different critical domains coincide. In higher
  dimensions the critical domains can overlap in several ways and therefore it is not evident how to choose an
  appropriate decomposition when this situation occurs. In both cases the affine subspace $S$ can be
  artificially and symbolically shifted into general position through the use of lexicographic perturbation.
\end{ex}
\begin{figure}[t]\label{fig:lexPerturbation}
  \centering
  \subfigure[\label{subfig:lexPerturbation1}]{\includegraphics[width=0.4\columnwidth]{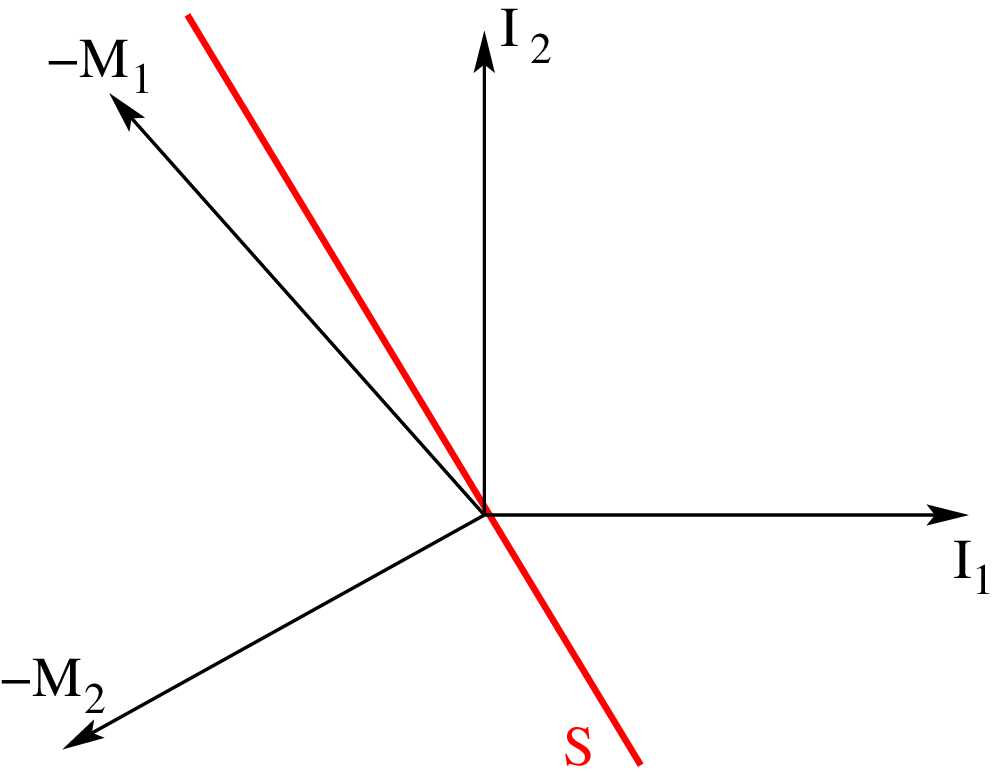}
    \includegraphics[width=0.4\columnwidth]{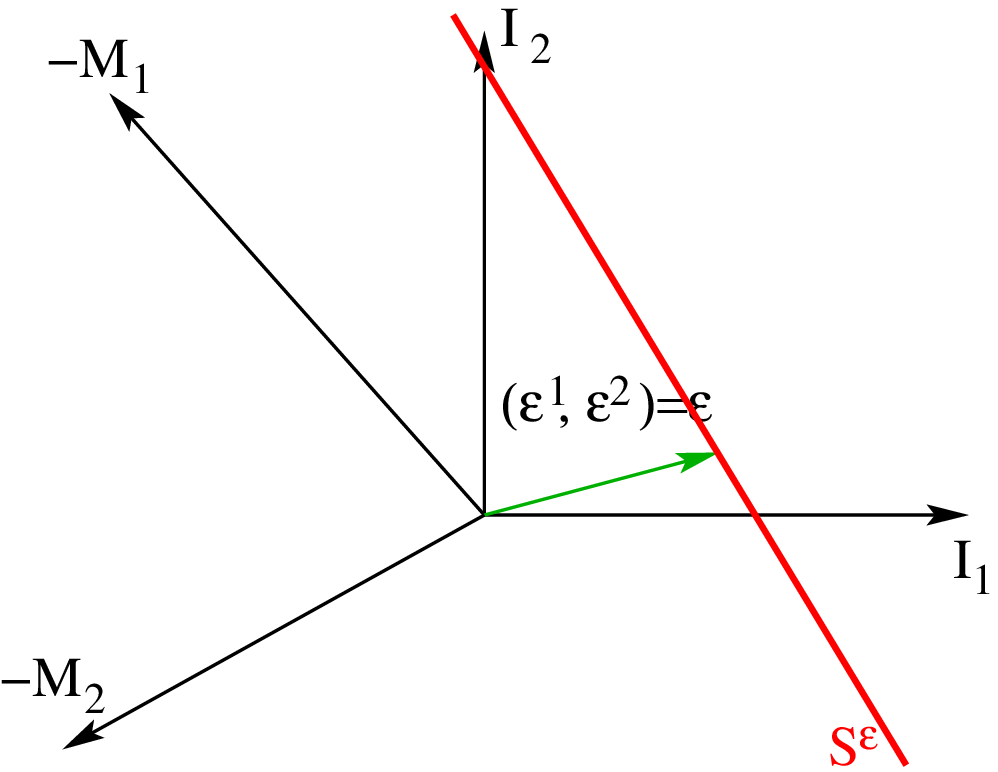}
  }\\
  \subfigure[\label{subfig:lexPerturbation2}]{\includegraphics[width=0.4\columnwidth]{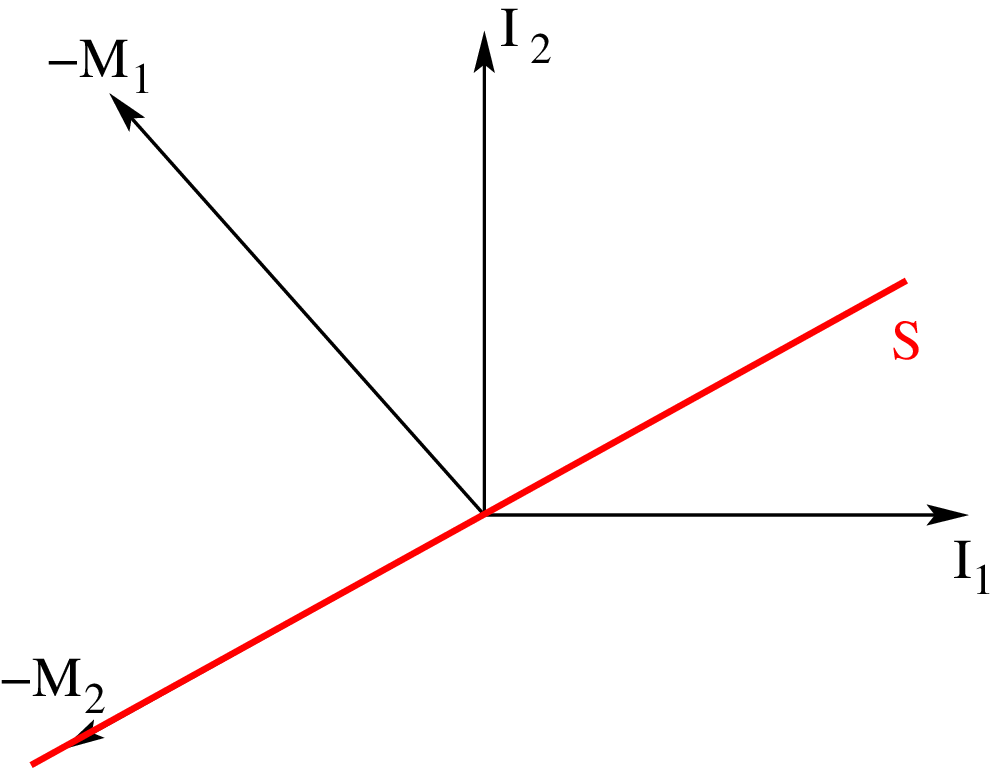}
    \includegraphics[width=0.4\columnwidth]{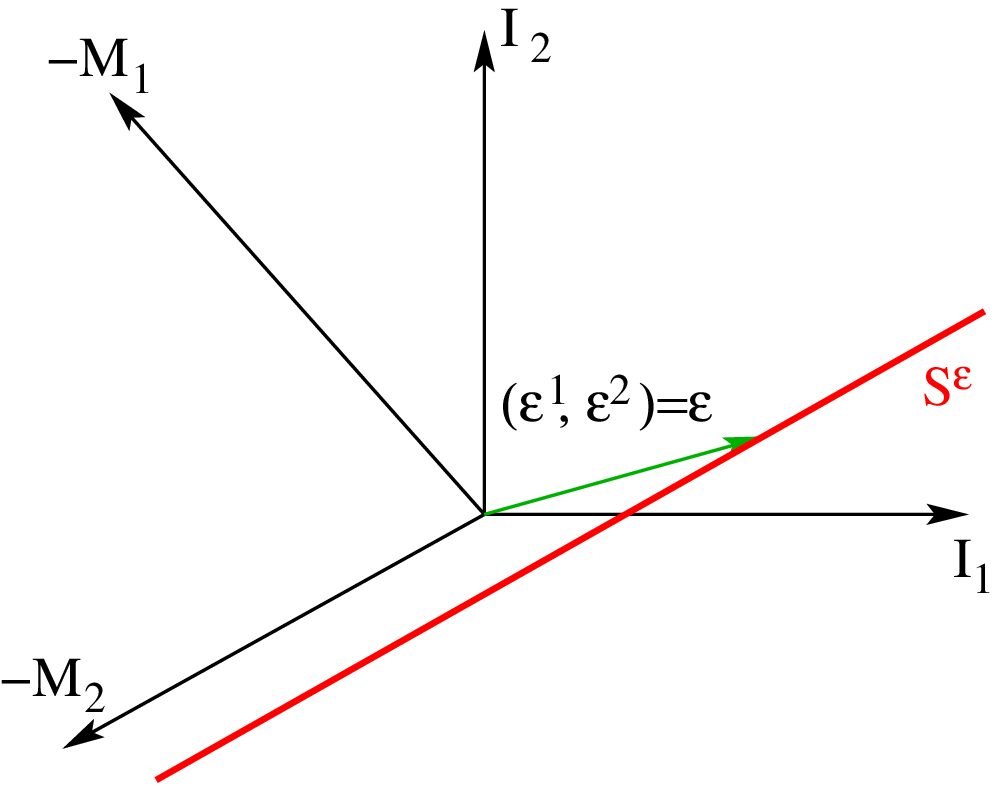}}
  \caption{Lexicographic perturbation of the affine subspace $S$}
\end{figure}

\subsection{Neighborhood computation in $S^\epsilon$}
\begin{algorithm}[t]
  \caption{function $\textit{neighbors}^\epsilon(B)$: returns all bases whose critical domains are adjacent to
    $S_B^\epsilon$.}\label{alg:neighbors_lex}
  \inputbox{A complementary basis $B$, a sufficient matrix $M$ and an affine subspace $S^\epsilon$.}\\
  \outputbox{The set of complementary bases $\BI$, whose critical domains are adjacent to $S^\epsilon_B$.}
  \begin{algorithmic}[1]
    \STATE $\BI:=\emptyset$%
    \STATE $\beta:=A\columnIndex{B}\inv$%
    \STATE $D:=-A\columnIndex{B}\inv A\columnIndex{N}\in \real^{B\times N}$%
    \COMMENT{where $N=\{1,\dots,2n\}\backslash B$}%
    \FOR{each $i\in B$}%
    \IF[if $\beta\rowIndex{i} y\geq 0$ is non-redundant]{$\textit{isLexPositive}(\redundancy(B,i))$ is
      true}\label{line:alg:redundancy_lex} \IF{$D_{i\bar i}>0$}%
    \STATE add $B':=B\backslash\{i\}\cup\{\bar i\}$ to $\BI$ 
    \ELSE
    \FOR {each $j\in B$ with $D_{i\bar j}<0$}
    \STATE{$B'':=B\backslash\{i,j\}\cup\{\bar i,\bar j\}$ }%
    \IF[if $dim(S^\epsilon_B \cap S^\epsilon_{B''})=d-1$]{$\textit{isLexPositive}(\adjacency(B,B''))$}\label{line:alg:adjacency_lex} %
    \STATE add $B''$ to $\BI$
    \ENDIF
    \ENDFOR
    \ENDIF
    \ENDIF
    \ENDFOR
    \RETURN{$\BI$}
  \end{algorithmic}
\end{algorithm}

Given a complementary basis $B$ that is feasible in $S^\epsilon$, the goal is to determine the adjacent
critical domains to $S^\epsilon_B$. Since $S^\epsilon$ lies in general position, it suffices to explore the
adjacent bases of the basis $B$. Similarly to the non degenerate case, we first determine the facets of
$S^\epsilon_B$ at Line~\ref{line:alg:redundancy_lex} of Algorithm~\ref{alg:neighbors_lex} and then compute the
adjacent critical domains that intersect with each facet.

Let $B$ be a complementary basis and consider $S^\epsilon_B=\{y\in\real^n\,\vert\, \,\beta y \geq 0 , \, y=
Q\theta + q+\epsLexVector\}$, where $\beta:=A\columnIndex{B}\inv $. The hyperplane $h_i:=\{y\,\vert\,
\beta\rowIndex{i}y=0 \}$, for some $i\in B$ intersected with $S^\epsilon_B$ forms a facet of $S^\epsilon_B$ if
there exists a $y^*\in S^\epsilon_B$ such that $\beta\rowIndex{j} y^*>0 $ for all $j\in B\backslash\{i\}$ and
$\beta\rowIndex{i}y^*=0$. Therefore $h_i\cap S_B^\epsilon$ is a facet of $S_B^\epsilon$ if and only if
\begin{equation}\label{eq:parametricLP_facet}
  \begin{array}{rcrclc}
    t^*(\epsilon)=&\text{max }& &t & &\\ 
    &\text{s.t.}& -\beta\rowIndex{j} Q\theta+t& \leq&\beta\rowIndex{j}(q+\epsLexVector),&\forall j\in B\backslash\{i\} \\
    & &-\beta\rowIndex{i}Q\theta& =& \beta\rowIndex{i}(q+\epsLexVector)&
  \end{array}
\end{equation}
has a positive optimal value for all $\epsilon>0$ sufficiently small.

This decision problem is no longer an LP, because the right hand side of the constraints depends on a
polynomial in $\epsilon$ and we want to know the behavior of $t^*(\cdot)$ in the neighborhood of zero.  In the
next subsection we will propose an efficient method for determining if $t^*(\cdot)$ is positive for
sufficiently small $\epsilon$.
 
If the hyperplane $h_i$ defines a facet of $S_B^\epsilon$, then we can distinguish the same three cases as
discussed in Section~\ref{subsection:Neighborhood_comput_GP_case}.
\paragraph{Diagonal pivot.}
If there is exactly one adjacent complementary basis $B'$, then $S_{B'}^\epsilon$ is full-dimensional (by the
general position of $S^\epsilon$) and is the unique adjacent critical domain to $S_B^\epsilon$ along
$S_B^\epsilon\cap h_i$.
\paragraph{Boundary of $\KAPPA(M)$.}
If there are no adjacent complementary cones to $\CI(B)$ along $h_i$ (see Corollary~\ref{cor:No_adj_cone}),
then there is no adjacent critical domain to $S_B^\epsilon$ along the facet $S_B^\epsilon\cap h_i.$
\paragraph{Exchange pivot.}
If there are adjacent bases $B''$ with $|B''\cap B|=2$ and $dim(\CI(B)\cap\CI(B''))=d-1$, then we must check
for each such basis $B''$ whether $dim(S_{B''}^\epsilon\cap S_{B}^\epsilon)=d-1$
(Line~\ref{line:alg:adjacency_lex} of Algorithm~\ref{alg:neighbors_lex}). Assume
$B'':=B\backslash\{i,j\}\cup\{\bar i,\bar j\}$ where $j\in B\backslash\{i\}$. As in the non-degenerate case,
we formulate a decision problem similar to LP~\eqref{eq:adjacencyCheck_exchange_pivot_non_degenerate} to test
the dimension of the intersection:
\begin{equation}\label{eq:parametricLP_adjacency}
  \begin{array}{rcrclc}
    t^*(\epsilon)=&\text{max }& & t & &\\ 
    &\text{s.t.}&\beta\rowIndex{k} Q\theta-t&\geq&-\beta\rowIndex{k}(q+\epsLexVector),& \forall k\in B\backslash\{i\} \\
    & &\beta\rowIndex{i}Q\theta& = &-\beta\rowIndex{i}(q+\epsLexVector)&\\
    & &\beta''\rowIndex{k} Q\theta-t& \geq& -\beta''\rowIndex{k}(q+\epsLexVector),& \forall k\in B''\backslash\{\bar j\} \\
    & &\beta\rowIndex{\bar j}Q\theta& = &-\beta\rowIndex{\bar j}(q+\epsLexVector),&
  \end{array}
\end{equation}
where $\beta:=A\columnIndex{B}\inv$ and $\beta'':=A\columnIndex{B''}\inv$.  The two critical domains are
adjacent, i.e. $\dim(S^\epsilon_B\cap S^\epsilon_{B''})=d-1$, in $S^\epsilon$ for all $\epsilon>0$
sufficiently small if and only if $t^*(\epsilon)>0$ for all sufficiently small
$\epsilon>0$. 

\bigskip
\subsubsection{Symbolic computation of the parametric Chebyshev center problem}
As seen in the previous subsection, the goal is to decide whether the optimal value $t^*(\epsilon)$
of~\eqref{eq:parametricLP_facet} (and of~\eqref{eq:parametricLP_adjacency}) is positive for all sufficiently
small $\epsilon>0$.  We call this decision problem a {\em parametric Chebyshev center problem}.  Here we
introduce a method that can compute exactly the behavior of $t^*(\cdot)$ for sufficiently small positive
$\epsilon$; the proposed approach is summarized as Algorithm~\ref{alg:solve_dual_lex_chebyshev}. The procedure
is explained only for~\eqref{eq:parametricLP_facet}, since the same method can be easily applied
for~\eqref{eq:parametricLP_adjacency}. The goal is to transform~\eqref{eq:parametricLP_facet} into a
multi-objective LP that can then be solved with any LP-solver. Recall the parametric LP
\eqref{eq:parametricLP_facet}:
\begin{equation}
  \label{eq:parametricLP_facet_bis}
  \begin{array}{rcrclc}
    t^*(\epsilon):=&\text{max }& t &\\ 
    &\text{s.t.}&-\beta\rowIndex{j} Q\theta+t&\leq&\beta\rowIndex{j}(q+\epsLexVector)\enspace,&\text{for all }j\in B\backslash\{i\} \\
    & &-\beta\rowIndex{i}Q\theta&=&\beta\rowIndex{i}(q+\epsLexVector)\enspace.&
  \end{array}
\end{equation}
For a fixed value of $\epsilon$, this is a linear program and its dual is:
\begin{equation}\label{eq:facet_dualLP_eps}
  \begin{array}{rcrclc}
    t^*(\epsilon)=&\text{min} &  (\beta(q+\epsLexVector))^T\,y& & & \\
    &\text{s.t.}& -(\beta \,Q)^Ty&=&0&\\ 
    & & \sum_{j\neq i }y_j&=&1&\\ 
    & & y_j&\geq&0&\text{for all }j\in B\backslash\{i\}\\
    & & y_i& & \text{free}\enspace.&
  \end{array}
\end{equation}
Let us denote its feasible region by $F(i)$, that is,
\begin{equation}
  F(i) = \{y \in \Real^n | -(\beta \,Q)^Ty = 0, \sum_{j\neq i }y_j=1, y_j \geq 0 \text{ for all } j\in B\backslash\{i\} \}.
\end{equation}

In order to solve~\eqref{eq:facet_dualLP_eps} symbolically we introduce the following standard notion.
\begin{defi}
  [Lexico-positive] A vector $a\in \real^s$ is lexico-positive (denoted by $a\succ 0$), if $a \not= 0$ and the
  first non-zero component of $a$ is strictly positive. Given two vectors $x$ and $y \in \real^s$, we write
  $x\succ y$ if and only if $x-y \succ 0$. A matrix is called lexico-positive if all its rows are
  lexico-positive. If $S=\{s^i\}_{i\in I}$ is a set of vectors, then $s^j$ is the lexico minimum of S if and
  only if $s^i\succeq s^j$ for each $i\in I$.
\end{defi}

The following theorem demonstrates that minimizing the polynomial cost function of~\eqref{eq:facet_dualLP_eps}
is equivalent to computing the lexicographic minimum of a vector.
\begin{thm}\label{thm:lexpositivenss1}
  If $y$ is a feasible vector of the dual problem
  (\ref{eq:facet_dualLP_eps}), 
  then the two statements below are equivalent:
  \begin{enumerate}
  \item $\exists \delta>0 $ such that $(\beta(q+\epsLexVector))^Ty>0$ for all $\epsilon\in(0,\delta)\enspace$,
  \item $(\beta[q\,I])^Ty\succ 0$\enspace.
  \end{enumerate}
\end{thm}
\begin{proof}
  The statement follows from the equality
  $(\beta(q+\epsLexVector))^T\,y=(1,\epsilon,\epsilon^2,\dots,\epsilon^n)(\beta[q\; I]) ^T y$, which holds for
  all $y$ and for all $\epsilon$. For every polynomial $p(x)=(1,x,x^2,\dots,x^n) (p_0,p_1,p_2,\dots,p_n)^T$
  the following holds: there exists a $\delta>0$ such that $p(x)>0$ for all $x\in(0,\delta)$ if and only if
  the first non-zero coefficient of $(p_0,p_1,p_2,\dots,p_n)$ is positive, i.e. $(p_0,p_1,p_2,\dots,p_n)$ is
  lexico positive.
\end{proof}
We can now consider an equivalent problem that we call a \emph{lexicographic linear program} (lexLP):
\begin{equation}\label{eq:lexfacet_dualLP_symb}
  \redundancy(B,i):
  \left\{\begin{array}{rcrc}
      T^*:= &\operatorname{lexmin} &  (\beta[q,\,I])^T\,y&\\
      &\text{s.t.}& y&\in F(i)
    \end{array}\right.
\end{equation}
The cost of this optimization problem is vector valued and the operator $\operatorname{lexmin}$ means to
compute the lexicographic minimum vector $(\beta[q,\,I])^Ty$ over all feasible decision variables $y$. We will
denote this particular lexLP, which tests the redundancy of the $i$-th inequality in $S_B^\epsilon$, as the
function $\redundancy(B,i)$.
\begin{thm}
  If $t^*(\cdot)$ is the optimal value of~\eqref{eq:parametricLP_facet_bis} as a function of $\epsilon$ and
  $T^*$ is the optimal value of~\eqref{eq:lexfacet_dualLP_symb}, then the following holds:
  \begin{equation}
    \exists \delta>0 \text{ such that } t^*(\epsilon)>0 \mbox{ for all } \epsilon\in(0,\delta)\iff T^*\succ 0\enspace.
  \end{equation}
\end{thm}
\begin{proof}
  The statement is a direct consequence of Theorem \ref{thm:lexpositivenss1}.
\end{proof}
Note that as is the case for linear programs, the restrictions and the objective function
of~\eqref{eq:lexfacet_dualLP_symb} are linear, although the objective returns a vector instead of a scalar. We
say that $T^*$ is the \emph{optimal value} of~\eqref{eq:lexfacet_dualLP_symb}.  If the vector $\beta_iQ$ is
non-zero then the feasibility region is bounded and therefore the optimal value is always attained if the
problem is feasible.

For our purposes, it is not necessary to compute the entire vector $T^*$, but only a sufficient number of its
elements in order to determine if it is lexico-positive or not. To this end, the goal is to find the first
non-zero component of $T^*$ and therefore the lex min problem~\eqref{eq:lexfacet_dualLP_symb} can be treated
as a multi-objective LP in the following way. First (say at step $0$) we solve the LP:

\begin{equation}\label{eq:dualLP_c_0}
  \begin{array}{rcrc}
    T_0:= &\text{min} &  c_0^T\,y&\\
    &\text{s.t.}& y& \in F(i),
  \end{array}
\end{equation}
where $c_0:= \beta\,q$. If $T_0\neq0$ then we can conclude that the optimal value $T^*$ of the
problem~\eqref{eq:lexfacet_dualLP_symb} is lexico-positive or lexico-negative from the sign of
$T_0$. Otherwise, if $T_0$ does equal zero, then we must consider the next objective function
$c_1:=\beta\rowIndex{1}$ and minimise it while maintaining $T_0=0$, and so on.

If $T_0=T_1=\dots=T_{r-1}=0$, then at the step $r$ we solve:
\begin{equation}\label{eq:dualLP_c_i}
  \begin{array}{rcrclc}
    T_r:= &\text{min} &  c_r^T\,y\\
    &\text{s.t.}& y& \in & F(i)\\ 
    & &c_k^T\,y&=&0,&k=0,\dots,r-1,\\
  \end{array}
\end{equation}
where $c_0=\beta\,q$ and $c_k=\beta\rowIndex{k}$ for $k=1,\dots,n$. If $T_r\neq0$ is the first non-zero value
of $T^*=(T_0,\dots,T_n)$ then $T^*$ of~\eqref{eq:lexfacet_dualLP_symb} is lexico-positive if $T_r > 0$ and
lexico-negative otherwise.  The resulting procedure is Algorithm \ref{alg:isLexPositive}, where the feasible
region of the LP (\ref{eq:dualLP_c_i}) is denoted by $F_r$.

\begin{algorithm}[H]
  \caption{Function $\textit{isLexPositive}(lexLP)$}\label{alg:solve_dual_lex_chebyshev}
  \label{alg:isLexPositive}
  \inputbox{A lex linear program $\textit{lexLP} $ as $\redundancy(\cdot,\cdot)$
    (\ref{eq:lexfacet_dualLP_symb}) or $\adjacency(\cdot,\cdot)$ (\ref{eq:lexadjacency_dualLP_symb}).}\\
  \outputbox{Answer about lexico-positiveness of the optimal value $T^*$ of~\eqref{eq:lexfacet_dualLP_symb}
    or~\eqref{eq:lexadjacency_dualLP_symb} respectively.}
  \begin{algorithmic}[1]
    \STATE Let $c_0,c_1, \dots,c_n$ be the objective functions of $\textit{lexLP} $ and $F_0$ be the
    feasibility region of $\textit{lexLP} $
    \FOR{$r=0$ to $n$}
    \STATE {$t_r^*:=\min\{c_r^Ty\,\vert\, \, y\in F_r\}$}\COMMENT{by solving the LP (\ref{eq:dualLP_c_i})} \IF
    { $t^*_r>0$ } \RETURN{$T^*$ is lexico-positive.}
    \ELSIF{$t^*_r<0$} \RETURN{ $T^*$ is lexico-negative.}  \ELSE \STATE $F_{r+1} := F_r\cap \{y\,\vert\,
    c_k^Ty=0\}$
    \ENDIF
    \ENDFOR
  \end{algorithmic}
\end{algorithm}

\begin{rem}
  Note that the lexLP $\adjacency(B,i)$ always has a non-zero optimal value $T^*$ because zero is not feasible
  in (\ref{eq:dualLP_c_i}) and the optimal solution $y^*$ of the last LP (\ref{eq:dualLP_c_i}), with $r=n$,
  must be optimal also for all previous LPs. Since $\beta$ has full rank we must have that $\beta y^*$ is
  non-zero and therefore there must be at least one component of $T^*$ that is non-zero.
\end{rem}

The parametric LP~\eqref{eq:parametricLP_adjacency} that determines whether two critical domains are adjacent
can also be solved using the same procedure. LP~\eqref{eq:parametricLP_adjacency} can be rewritten as a lex
min LP as follows:
\begin{equation}\label{eq:lexadjacency_dualLP_symb}
  \adjacency(B,B''):
  \left\{
    \begin{array}{rcrclc}
      T^*:= &\text{lexmin} &  (\beta[q,\,I])^T\,y+ (\beta''[q,\,I])^T\,x&\\
      &\text{s.t.}& -(\beta \,Q)^Ty-(\beta'' \,Q)^Tx&=&0&\\ 
      & & \sum_{k\in B\backslash\{ i\} }y_k+\sum_{k\in B''\backslash\{ \bar j\} }x_k&=&1&\\ 
      & & y_k&\geq&0,&k\in B\backslash\{i\}\\
      &  & x_k&\geq&0,&k\in B''\backslash\{\bar j\}\\
      &  & y_i ,x_{\bar j}& &\text{free}.& 
    \end{array}\right.
\end{equation}
We will call this lexLP $\adjacency(B,B'')$ because it tests the adjacency of $S_B$ and $S_{B''}$.  Recall
that one of the two equalities in~\eqref{eq:parametricLP_adjacency} can be removed because one is redundant
and therefore one of $y_i$ and $x_j$ is also redundant. Hence,~\eqref{eq:lexadjacency_dualLP_symb} has the
same structure as~\eqref{eq:lexfacet_dualLP_symb} and can be solved as a multi-objective LP as explained above
(see Algorithm~\ref{alg:solve_dual_lex_chebyshev}).

\subsection{Post-processing of the graph of critical domains $\GI^\epsilon$}
The previous section introduced a computational method for computing the graph of critical
domains~$\GI^\epsilon$ relative to the lex-perturbed space $S^\epsilon$ for all $\epsilon>0$ sufficiently
small. The goal in this section is to recover the graph $\GI=(V,E)$ of critical domains relative to the
original space $S$ according to Definition~\ref{def:graph_of_CD}.

The following theorems will show that one can construct from $\GI^\epsilon$ the graph $\GI$ of critical
domains relative to the unperturbed space $S$.
\begin{thm}\label{thm:lexPerturbation}
  Let $M$ be a sufficient matrix, $S$ an affine subspace and consider the graph of critical domains
  $G^\epsilon=(V^\epsilon,E^\epsilon)$ relative to the lexicographically perturbed space $S^\epsilon$ for
  sufficiently small $\epsilon>0$.  Let $V$ be the set of all bases $B$ in $V^\epsilon$ with full-dimensional
  critical domains $S_B$.  Then the following statements hold.
  \begin{enumerate}
  \item For each $B\in V^\epsilon$, the complementary cone $\CI(B)$ intersects $S$ and therefore $S_B$ is
    nonempty.
  \item For any two distinct bases $B_1$ and $B_2$ in $V$, $S_{B_1}$ and $S_{B_2}$ have disjoint relative
    interiors.\label{item:disjoint_rel_interior}
  \item The set of all full-dimensional critical domains $S_{B}$ for $B\in V$ covers $S_f$, i.e.,
    $$\bigcup_{B \in V}S_{B}=S_f = \KAPPA(M)\cap S\enspace,$$
    and forms a polyhedral decomposition of $S_f$.
  \end{enumerate}
\end{thm}
\begin{proof}
  First, note that complementary cones are closed and so the first statement follows directly.

  To prove the second, note that for all $\epsilon>0$ sufficiently small (say $\epsilon<\delta$) $S^\epsilon$
  lies in general position. From Proposition~\ref{prop:adjcencyCD_CC} the critical domains defined by
  $S^\epsilon$ are disjoint in their interiors for any positive $\epsilon<\delta$. The second statement
  follows from the fact that for any basis $B$, $S^\epsilon\cap \CI(B)$ changes continuously in $\epsilon$,
  for $\epsilon< \delta$.

  The third statement is proven in two steps. First we prove that the critical domains whose bases are in
  $V^\epsilon$ define a covering of $S\cap \KAPPA(M)$. Let $q\in \KAPPA(M)\cap S $, since $\epsLexVector\in
  K(M)$ and $\KAPPA(M)$ is a convex cone, there exists a basis $B$ such that $q+\epsLexVector\in \CI(B)$ for
  all $\epsilon>0$ sufficiently small. Therefore $q$ is in $ \CI(B)$ and hence in $S_B$ because complementary
  cones are closed and thus $\bigcup_{B \in V^\epsilon}S_{B}= S_f=\KAPPA(M)\cap S $. Since critical domains
  are closed and $S_f$ is a convex polyhedron, the full-dimensional critical domains define a covering of
  $S_f$.
\end{proof}

The above theorem demonstrates that the bases in $V^\epsilon$ have nonempty critical domains in the
unperturbed space $S$. Moreover, there exists a subset $V$ whose critical domains form a polyhedral
decomposition of $S_f$ according to Definition~\ref{def:graph_of_CD}. The next theorem discusses how adjacency
in $\GI^\epsilon$ relates to adjacency in $\GI$.
\begin{thm}\label{thm:adjacency_path}
  Let $M$ be a sufficient matrix and $B_1, B_2$ be two complementary bases in $V\subset V^\epsilon$,
  i.e. $S_{B_1}$ and $S_{B_2}$ both have dimension $d$. If $S_{B_1}$ and $S_{B_2}$ are adjacent, then there
  exists a path $\tilde B^1,\tilde B^2,\dots,\tilde B^r$ in $\GI^\epsilon=(V^\epsilon,E^\epsilon)$ from $B_1$
  to $B_2$ with the following property: $S_{\tilde B^i}$ intersects $S_{B_1}\cap S_{B_2}$ with dimension
  $d-1$, i.e. $\dim(S_{\tilde B^i}\cap S_{B_1}\cap S_{B_2})=d-1$ for all $i=1,\dots,r$.

\end{thm}
\begin{proof}
  If $S_{B_1}^\epsilon$ and $S_{B_2}^\epsilon$ are adjacent, $(B_1,B_2)$ is clearly the desidered path. We
  assume they are not adjacent.  We choose a $\bar q \in \relINT(S_{B_1}\cap S_{B_2})$ which is not contained
  in any critical domain or in any face of dimension $d-2$ or less, and let $\bar \theta \in \real^d$ be the
  parameter with $\bar q=q+Q\bar \theta$.  We look now (for a moment) at the parameter space and at the
  critical regions. The hyperplane $f:=\{\theta\, \vert\, a^T\theta=b\}$ contains the intersecion of the two
  critical regions, i.e. $f\supseteq (\ERRE_{B_1}\cap \ERRE_{B_2})$ and consider its perpendicular (normal?)
  line $\theta(t)=\bar \theta+t a$. The image of $\theta(t)$ is $q(t)=\bar q + t Q a$ in the original space
  $S$, respectively $q^\epsilon(t)=\bar q+\epsLexVector + t Q a$ in the perturbed space $S^\epsilon$.

  We know that for each $\epsilon>0$ sufficiently small every critical domain becomes either full-dimensional
  or empty. The full-dimensional ones vary continuously in function with $\epsilon$. Consider a segment
  $[q^\epsilon(t_1),q^\epsilon(t_2)]$ of the line $\{q^\epsilon(t)~\vert~ t\in\Real\}$ such that it intersects
  either $S_{B_1}^\epsilon$ and $S_{B_2}^\epsilon$ for all $\epsilon>0$ sufficiently small. Because of the
  continuity no critical domain, which has dimension smaller than $d-1$ in the original space $S$ intersects
  this segment for all $\epsilon>0$ sufficiently small. Similarly, for any $B\in V^\epsilon$ no face of
  $S_B^\epsilon$ of dimension smaller than $d-1$ intersects $[q^\epsilon(t_1),q^\epsilon(t_2)]$ for all
  $\epsilon>0$ sufficiently small.  The desidered path is given by the critical domains which decompose the
  line segment between $S_{B_1}$ and $S_{B_2}$.
\end{proof}

Note that the last condition in the above theorem, along with Proposition~\ref{prop:disjointCD}, implies that
$dim(S_{\tilde B^i})=d-1$ for $i=2,\dots,r-1$ and therefore Theorems~\ref{thm:lexPerturbation}
and~\ref{thm:adjacency_path} imply the following corollary.
\begin{cor} \label{cor:postadjacency} Let $M$ be a sufficient matrix, $S$ an affine subspace and let
  $\GI^\epsilon=(V^\epsilon,E^\epsilon)$ be the graph of critical domains relative to $S^\epsilon$. Then, the
  graph of critical domains $\GI=(V,E)$ relative to $S$ is related to $\GI^\epsilon$ as follows:
  \begin{enumerate}
  \item $V\subseteq V^\epsilon$
  \item For every basis $B\in V$, $S_B$ has dimension $d$.
  \item For each pair of bases $B_1$ and $B_2$, the critical domains $S_{B_1}$ and $S_{B_2} $ are adjacent if
    and only if there exists a path $(\tilde B_1,\dots ,\tilde B_r)$ in $\GI^\epsilon$ with $\tilde B_1=B_1$,
    $\tilde B_r=B_2$ and $dim(S_k)=d-1$ for $k=2,\dots,r-1$ (or $(B_1,B_2)\in E^\epsilon$ ).
  \end{enumerate}
\end{cor}

The above corollary provides a simple procedure for computing a critical region graph $\GI=(V,E)$ relative to
the unperturbed affine set $S$ from the perturbed one $\GI^\epsilon$. We begin from the perturbed critical
region graph $\GI = (V,E) := \GI^\epsilon$ and remove each node $B$ from $\GI$ that has a critical domain
$S_B$ which is not full-dimensional and add all new edges $(B_1,B_2)$ to $E$ satisfying the statement 3 of
Corollary \ref{cor:postadjacency}. From Theorem~\ref{thm:lexPerturbation}, the critical domains of the nodes
of the resulting graph will form the desired polyhedral covering of the
$S_f$. Theorem~\ref{thm:adjacency_path} states that the resulting graph contains edges for all adjacent bases,
but may be overconnected since some of the critical domains $S_B$ of removed bases $B$ may have had a
dimension less than $d-1$. It remains, therefore, to test each edge in order to determine if the connected
bases are in fact adjacent in the unperturbed space. As discussed previously, both operations for testing
full-dimensionality and adjacency can be posed as linear programs.

\begin{rem}
  Note that much of the computation required to test for full-dimensionality of the critical domains for the
  unperturbed affine set has already been done while building the perturbed graph. Specifically, one can
  determine if a region is full-dimensional by examining the first component $T_0$ of the optimizer of
  LP~\eqref{eq:lexfacet_dualLP_symb}.
\end{rem}

\section{Complexity of the algorithm}\label{section:complexity}
In this section we will discuss the complexity of the proposed algorithm, which enumerates all
full-dimensional critical domains relative to the lexicographically perturbed affine subspace
$S^\epsilon$. The well-known example by Murty (see~\cite{murty78} or see Chapter~6
in~\cite{murty_onlinebook}), which was used to prove the non-polynomiality of the Lemke and the principal
pivoting methods, can be easily seen to demonstrate that the number of critical domains of a pLCP with an
affine subspace $S$ of dimension $1$ and $\PI$-matrix $M$ is exponential in $n$.  Since the complexity of the
graph search (Algorithm~\ref{alg:graph_search}) is a polynomial function of the number of critical domains, no
algorithm for pLCP is polynomial in $n$.  It is, however, possible to bound the number of operations required
to explore the neighborhood of each critical domain, i.e. the complexity of Algorithms~\ref{alg:neighbors}
and~\ref{alg:neighbors_lex}. Since each critical domain will be explored exactly once, we can say that the
algorithm is output sensitive in that its complexity is a polynomial function of the number of
full-dimensional critical domains and the size of input, provided that a polynomial-time algorithm for linear
programming is used.

We first consider the general position case and study the complexity of Algorithm~\ref{alg:neighbors}. Assume
that $M$ is of order $n$, the affine subspace $S$ is of dimension $d$ and let $B$ be a complementary basis
with a nonempty critical domain $S_B$.  The main computations of the function $neighbors(B)$
(Algorithm~\ref{alg:neighbors}) are:
\begin{itemize}
\item Redundancy checking at Line~\ref{line:alg:redundancy}\\
  (Solve LP~\eqref{eq:facet_dualLP_eps} with $n$ variables and $d+1$ constraints.)
\item Checking adjacency for the case of an exchange pivot at Line~\ref{line:alg:adjacency}\\
  (Solve LP~\eqref{eq:adjacencyCheck_exchange_pivot_non_degenerate} with $2n$ variables and $d+1$
  constraints.)
\end{itemize}
We denote the time necessary to solve an LP in standard form by $\timeLP{\textit{var}}{\textit{eq}}$, where
$\textit{var}$ denotes the number of (nonnegative) variables and $\textit{eq}$ is the number of equality
constraints. The time necessary to explore a critical domain can then be bounded as follows.
\begin{thm}
  Let $M\in\real^{n \times n}$ be a sufficient matrix and assume that the affine subspace $S$ lies in general
  position. For each complementary basis $B$ with a nonempty critical domain $S_B$ the time necessary to
  explore the neighborhood of $S_B$ is bounded by:
  \begin{equation}\label{eq:bound_LP_GP}
    n \, \timeLP{n}{d+1} + \frac{n^2-n}{2} \, \timeLP{2n}{d+1}.
  \end{equation}
\end{thm}
\begin{proof}
  Redundancy checking requires the solution of LP~\eqref{eq:facet_dualLP_eps} once for each of the $n$
  inequalities of $S_B$, which takes $n \, \timeLP{n}{d+1}$ time. Adjacency checking by solving the
  LP~\eqref{eq:adjacencyCheck_exchange_pivot_non_degenerate} is necessary only in the case that the considered
  adjacent basis $B''$ differs by two elements from the basis $B$ and since there are at most
  $\frac{n^2-n}{2}$ such bases, the second term $\frac{n^2-n}{2} \, \timeLP{2n}{d+1}$ follows.
\end{proof}
If $S$ does not lie in general position, then the lexLPs (\ref{eq:lexfacet_dualLP_symb}) and
(\ref{eq:lexadjacency_dualLP_symb}) are solved instead of LPs~\eqref{eq:LP_facet}
and~\eqref{eq:adjacencyCheck_exchange_pivot_non_degenerate}. Each lexLP can be solved as a sequence of at most
$n+1$ LPs with the same variables and constraints (see Algorithm~\ref{alg:solve_dual_lex_chebyshev}), which
leads to the following complexity bound.
\begin{thm}
  If $M\in\real^{n \times n}$ is a sufficient matrix, then for each complementary basis $B$ with nonempty
  critical domain $S^\epsilon_B$ the time necessary to explore the neighborhood of $S^\epsilon_B$ can be
  bounded by
  \begin{eqnarray}\label{eq:bound_LP_nonGP}
    (n^2+n)\, \timeLP{n}{d+1} + \frac{n^3-n}{2} \,\timeLP{2n}{d+1}.
  \end{eqnarray}
\end{thm}
The above theorems bound the complexity of ``exploring'' a basis of the output in
Algorithm~\ref{alg:graph_search} (Line~\ref{line:exploration}). The condition at
Line~\ref{line:explored_condition}, which can be verified in time bounded by the logarithm of the size of the
output, ensures that each output basis is explored exactly once. As a result, the complexity of the algorithm
grows linearly with the size of the output and so is output sensitive.

\section{Example}
In this section we present a simple illustrative example that arises from control theory. Consider the
following discrete time constrained linear time-invariant system:
\begin{align*}
  x^+ = \begin{pmatrix}1 & 1\\0 & 1
  \end{pmatrix} x + \begin{pmatrix} 1 \\ 0.5
  \end{pmatrix}u\enspace,   \\
  ||x||_\infty \leq 5,~~~||u||_\infty \le 1\enspace,
\end{align*}
where $x \in \real^2$ is the system state, $x^+$ is the successor state and $u \in \real$ is the system
input. A common method of control for this class of systems is Model Predictive Control, in which we solve at
each point in time the following finite horizon optimal control problem:
\begin{align}\label{ex:mpc}
  J^\star (\theta) = &\min \sum_{k=0}^{N-1}\limits\begin{pmatrix} \norm{Q x_k}_2^2 + \norm{R
      u_k}_2^2 \end{pmatrix}
  + \norm{Q_f x_N}_2^2\nonumber\\
  & \text{subject to}\nonumber\\
  &\begin{aligned} &x_{k+1} = \begin{pmatrix}1 & 1\\0 & 1
    \end{pmatrix} x_{k} + \begin{pmatrix} 1 \\ 0.5
    \end{pmatrix}u_{k}   \\
    &||x_k||_\infty \leq 5\enspace,~~~||u_{k-1}||_\infty \le 1\enspace,~~~~\forall k\in\left\{1,\dots,N\right\}\enspace\\
    &x_0 = \theta\enspace,
  \end{aligned}
\end{align}
where $\theta$ is the current state of the system, the prediction horizon $N$ is $5$ and the weighting
matrices $Q$, $R$ and $Q_f$ are the identity. For high-speed systems, such as electric power converters, the
goal is to solve the above quadratic program as rapidly as possible, in some cases at rates exceeding hundreds
of kilohertz~(e.g.~\cite{BecEtal:2007:IFA_2907}). By computing the optimizer offline as an explicit
piecewise-affine function of the state $\theta$, these speeds can be
achieved~\cite{seron:goodwin:dona:2000,johansen:peterson:slupphaug:2000,BMDP02a}. The above parametric
quadratic program is easily converted to a pLCP with a positive semi-definite matrix
$M$~\cite{murty_onlinebook}, which was then solved using the proposed algorithm. The resulting polyhedral
partition and mapping from the parameter $\theta$ to the optimizer $u_0$ is shown in
Figure~\ref{fig:doubleInt}.

\begin{figure}
  \centering \resizebox{0.48\columnwidth}{!}{\includegraphics{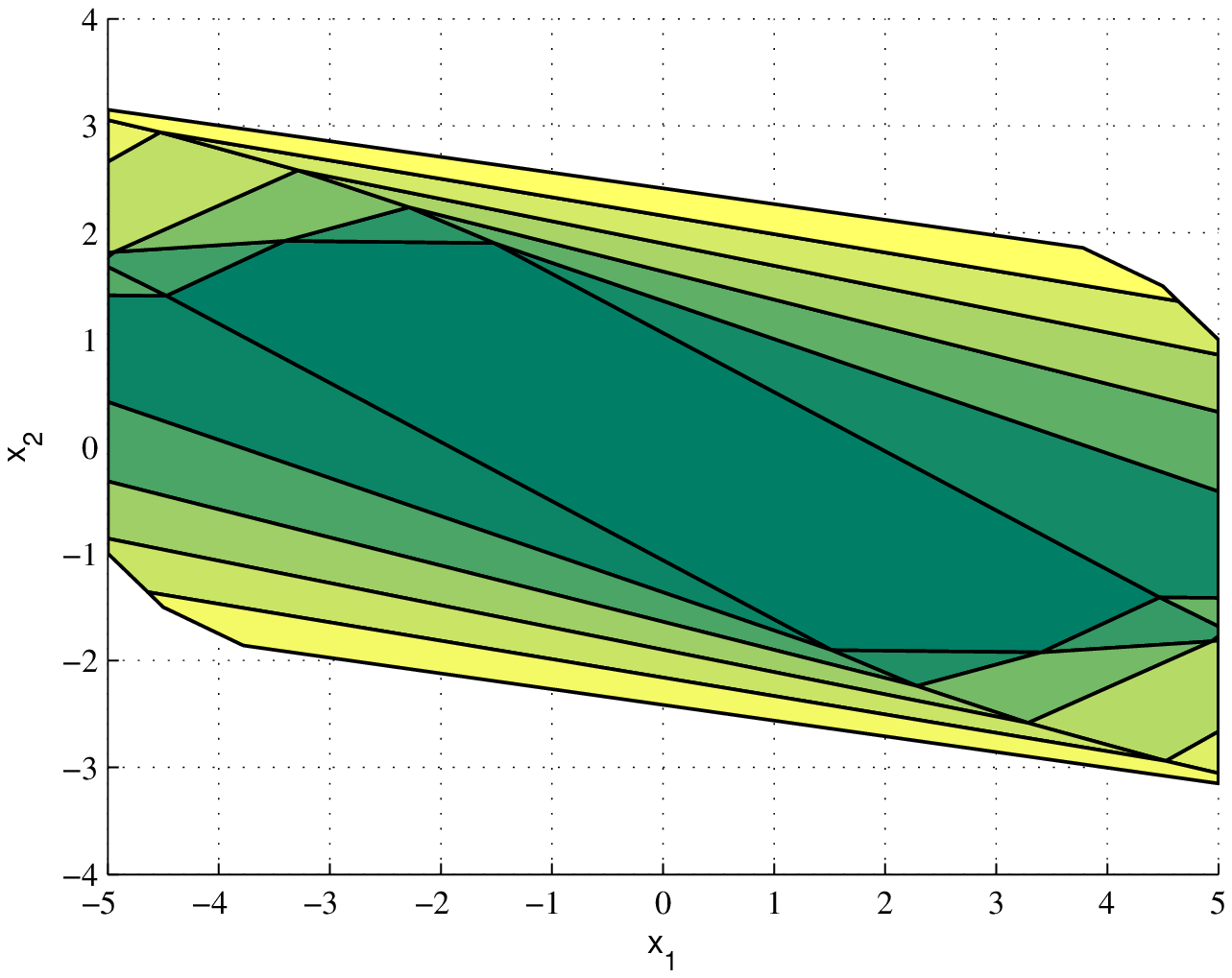}}
  \resizebox{0.48\columnwidth}{!}{\includegraphics{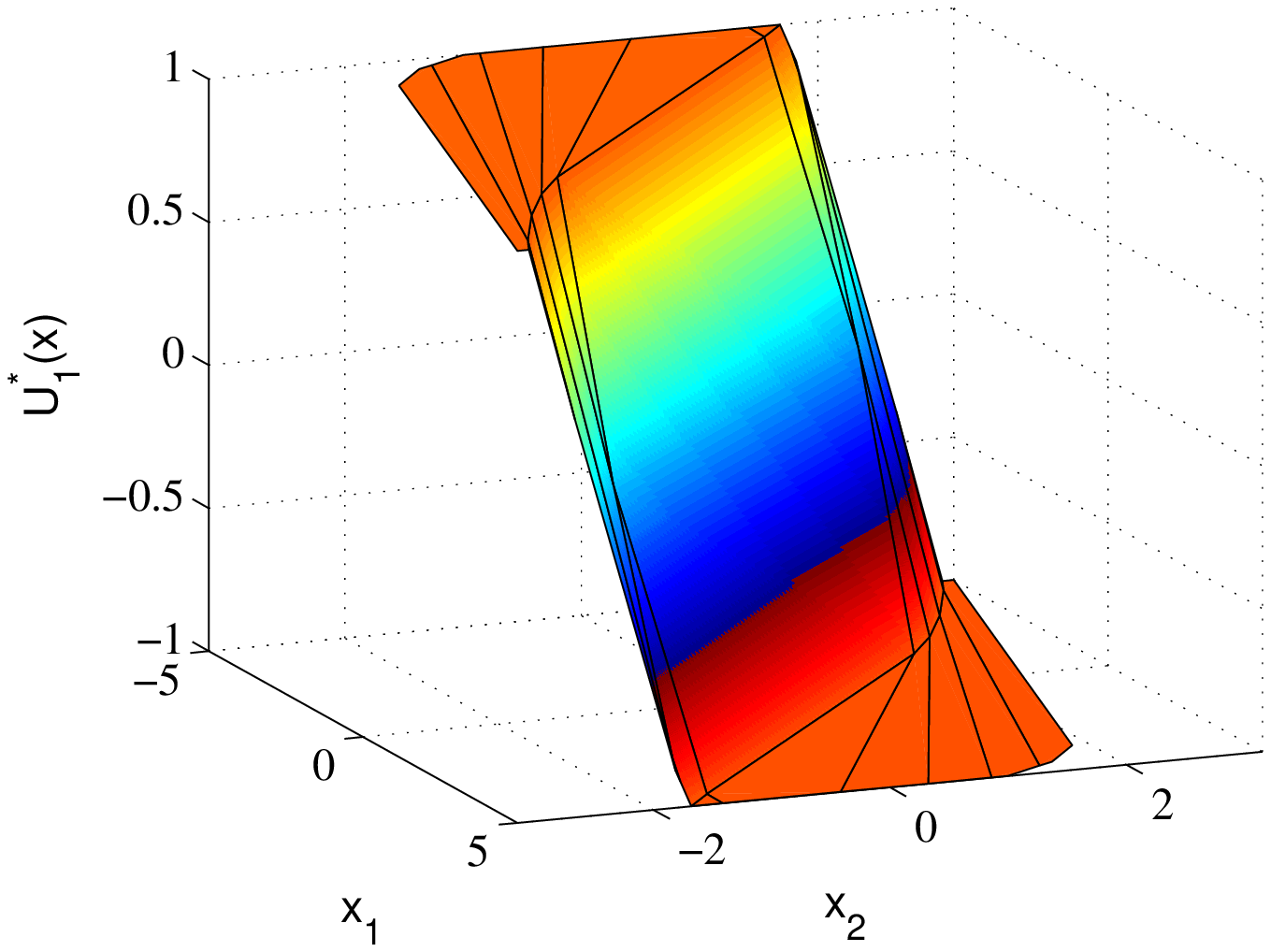}}
  \caption{Polyhedral partition (left) and optimal input (right) for the control problem~\eqref{ex:mpc}}
  \label{fig:doubleInt}
\end{figure}

\section{Conclusion}
In this paper an algorithm to enumerate all feasible bases of the parametric LCP defined by a sufficient
matrix $M$ and a lexicographically perturbed affine subspace $S$ was proposed. It has been shown that the
perturbed parametric LCP can be solved in a time linearly bounded by the size of the output and moreover, this
output can be efficiently post-processed in order to generate a polyhedral decomposition for the unperturbed
original affine subspace $S$.

One feature of the algorithm which is not ideal is the space requirement.  Namely, the proposed algorithm must
store all discovered feasible bases in the memory because it relies on the standard graph search technique.  A
great improvement can be made if we could apply the reverse search technique \cite{af-rse-96} which is
essentially memory free.  For this, it is necessary for the underlying graph to be oriented properly with
exactly one sink.  Somewhat similar to the present work, the paper \cite{fjt-cgf-07} proposed an algorithm to
compute a polyhedral complex known as the Gr\"obner fan which was shown to have such a ``reverse search
property.''  Finding such an orientation for the graph of critical domains is an excellent subject of the
future research.

\bibliographystyle{plain} \bibliography{paper}

\appendix
\section{Useful properties of matrix classes}\label{APPENDIX}
This section gives an overview of some matrix classes with important properties for linear complementarity
problems. The reader is referred to~\cite{cottle} for a thorough survey.\\

\subsection*{$\PI$-Matrices}
\begin{defi}
  The matrix $M\in \real^{n\times n}$ is a $\PI$-matrix if and only if all principal minors of $M$ are
  strictly positive.
\end{defi}
This class characterizes the matrices $M$ for which the corresponding LCP always has a unique solution.
\begin{thm}
  The following statements are equivalent:
  \begin{enumerate}
  \item $M\in \PI$,
  \item The LCP defined by the matrix $M$ has a unique solution for all right hand side vectors
    $q\in\real^{n}$,
  \item $M$ does not reverse the sign of any nonzero vectors, i.e. $$[z_i(Mz)_i\leq0 \mbox{ for all }
    i]\Rightarrow[z=0]\enspace.$$
  \end{enumerate}
\end{thm}
Recall that $M\in \real^{n\times n}$ is a positive definite matrix, denoted with $PD$ if for all $x\in\real^n$
it holds that $x^TMx>0$. It is then easy to see from the above theorem that positive definite matrices belong
to the class $\PI$.

\subsection*{$\PI_0$-matrices}
\begin{defi}
  The matrix $M \in \real^{n\times n}$ is a $\PI_0$-matrix if and only if all principal minors of $M$ are
  non-negative.
\end{defi}
Analogously to the positive-definite case above, positive-semidefinite matrices (PSD) are clearly in
$\PI_0$. The following theorem gives properties of PSD matrices relevant to the solution of pLCPs.
\begin{thm}\label{thm:P_0_characterisation}
  Let $M\in \real^{n\times n}$ be a matrix and $I$ be the identity matrix of the same order. The following
  statements are equivalent:
  \begin{enumerate}
  \item $M\in \PI_0$,
  \item For each vector $x\neq 0$ there exists an index $k$ such that $z_k\neq0$ and $z_k(Mx)_k\geq0$,
  \item $(M+\epsilon I)$ is a $\PI$ matrix for all $ \epsilon >0 $.
  \end{enumerate}
\end{thm}

\subsection*{Semimonotone matrices}
\begin{defi}
  A matrix $M\in \real^{n\times n}$ is called \emph{semimonotone} if the following holds:
  \begin{equation}
    \text{for all } x\geq 0, x\neq 0 \Rightarrow [x_k>0 \text{ and } (Mx)_k\geq0 \text{ for some }k]\enspace.
  \end{equation}
\end{defi}
The class of such matrices is denoted by $E_0$ and by Theorem~\ref{thm:P_0_characterisation}, every
$\PI_0$-matrix is semimonotone.
\begin{defi}\label{def:fullysemimonotone}
  Let $M\in \real^{n\times n}$ be a semimonotone matrix.  If for all index subsets $\alpha \subseteq
  \{1,\dots,n\}$ with $det(M_{\alpha \alpha})\neq 0$ the \emph{principal pivot transform of $M$ with respect
    to $\alpha$}
  \[M':=\begin{bmatrix}
    M^{-1}_{\alpha \alpha} & - M^{-1}_{\alpha \alpha} M_{\alpha \bar{\alpha}}\\
    M_{\bar{\alpha} \alpha } M^{-1}_{\alpha \alpha} & M^{-1}_{\bar \alpha \bar \alpha} - M_{\bar{\alpha}
      \alpha } M^{-1}_{\alpha \alpha} M_{\alpha \bar{\alpha}}
  \end{bmatrix}\] is semimonotone, then $M$ is called \emph{fully semimonotone}. The class of such matrices is
  denoted with $\fullySemimo$ and $M$ is said to be an $\fullySemimo$-matrix.
\end{defi}

\subsection*{$\QU_0$-Matrices}
\begin{defi}
  An LCP defined by the matrix $M$ and right hand side vector $q$ is called \emph{weakly feasible} if there
  exist positive vectors $z$ and $w$ such that $w-Mz=q$ and \emph{feasible} if $z'w=0$ also holds. The class
  of matrices $M$ for which the LCP is feasible whenever it is weakly feasible, is denoted by $\QU_0$.
\end{defi}
Since $\PI$-matrices are feasible for each vector $q$, they are also $\QU_0$-matrices.

\begin{thm}
  Let $M\in \real^{n\times n}$ and $I$ be the identity matrix of same order. The following statements are
  equivalent:
  \begin{enumerate}
  \item $M\in \QU_0$,
  \item The complementary range $\KAPPA(M)$ is convex,
  \item $\KAPPA(M)=\cone([I\; {-M}])$
  \end{enumerate}
\end{thm}
The implications of convexity of the complementary range are discussed in the next section.

\subsection*{Sufficient Matrices}
\begin{defi}
  A square matrix $M$ is called \emph{column sufficient} if it satisfies the implication:
  \begin{equation}
    [z_i(Mz)_i\leq0 \mbox{ for all }i]\; \Longrightarrow \;[z_i(Mz)_i=0\mbox{ for all }i]\enspace.
  \end{equation}
  The matrix $M$ is called \emph{row sufficient} if its transpose is column sufficient. If $M$ is both column
  and row sufficient, then it is said to be \emph{sufficient}.
\end{defi}
\begin{thm} \label{thm:rowsuff P0Q0} If $M$ is row sufficient matrix, then
  \begin{enumerate}
  \item $M\in \PI_0$,
  \item $M\in \QU_0$.
  \end{enumerate}
\end{thm}
From the theorem above we have that every column sufficient matrix also belongs to $\PI_0$. Below we state a
characterisation of column sufficient matrices, which has an important implication regarding the structure of
the resulting complementary cones.
\begin{thm}\label{thm:characterisation_CSu}
  Given a matrix $M\in \real^{n\times n}$, the following statements are equivalent:
  \begin{enumerate}
  \item $M$ is column sufficient,
  \item For each vector $q\in \real^n$ the following holds: if $(z^1,w^1), (z^2,w^2) $ are two solutions of
    the LCP defined by the matrix $M$ and the vector $q$, then $(z^1)^Tw^2=(z^2)^Tw^1=0$.
  \end{enumerate}
\end{thm}

\end{document}